\documentclass[12pt,lot, lof]{amsart}
\author[Mohammad F.Tehrani]{Mohammad  Farajzadeh Tehrani}
\address{Department of Mathematics, Princeton
University}
\email{mfarajza@math.princeton.edu}

\title[Open GW theory on symplectic manifolds and symplectic cutting]
{Open Gromov-Witten theory on symplectic~manifolds and symplectic cutting}

\usepackage{epsfig}
\usepackage{amsfonts,amssymb,amsmath}   
\usepackage{amsthm}
\usepackage{latexsym}
\usepackage{graphicx}
\usepackage[all]{xy}
\usepackage{verbatim}
\usepackage{multirow}
\usepackage{longtable}
\usepackage{booktabs}
\usepackage{hyperref}
\hypersetup{bookmarksnumbered}
\hypersetup{colorlinks,bookmarksnumbered}
 \usepackage{enumitem}

\numberwithin{equation}{section}

\newtheorem{theorem}{Theorem}[section]
\newtheorem{lemma}[theorem]{Lemma}
\newtheorem{corollary}[theorem]{Corollary}
\newtheorem{proposition}[theorem]{Proposition}

\theoremstyle{definition}
\newtheorem{definition}[theorem]{Definition}
\newtheorem{remark}[theorem]{Remark}
\newtheorem{example}[theorem]{Example}

\def\R{\mathbb R}
\def\C{\mathbb C}
\def\Z{\mathbb Z}
\def\Q{\mathbb Q}
\def\P{\mathbb P}
\def\H{\mathbb H}

\def\mc{\mathcal}
\def\la{\lambda}
\def\ep{\epsilon}
\def\ra{\rightarrow}
\def\De{\Delta}
\def\de{\delta}
\def\om{\omega}
\def\si{\sigma}
\def\Si{\Sigma}
\def\La{\Lambda}
\def\ze{\zeta}

\def\ov#1{\overline{#1}}
\def\tn#1{\textnormal{#1}}
\def\mf#1{\mathfrak{#1}}

\def\Eq#1{\begin{equation}\label{#1}}
\def\EEq{\end{equation}}
\def\sEq#1{\begin{equation*}\label{#1}}
\def\sEEq{\end{equation*}}
\def\Thm#1{\begin{theorem}\label{#1}}
\def\EThm{\end{theorem}}
\def\Lem#1{\begin{lemma}\label{#1}}
\def\ELem{\end{lemma}}
\def\Prop#1{\begin{proposition}\label{#1}}
\def\EProp{\end{proposition}}

\def\Fix{\tn{Fix}}
\def\ev{\tn{ev}}
\def\id{\tn{id}}
\def\vir{\tn{vir}}

\def\index{\tn{ind}}
\def\disc{\tn{disc}}
\def\FS{\tn{FS}}

\def\cut{\tn{cut}}
\def\td{\tn{d}}

\begin{document}
\begin{abstract}
Let $(X,\om)$ be a symplectic manifold and $L$ be a Lagrangian
submanifold diffeomorphic to $S^n$, $\R\P^n$, or a Lens space of a certain
type. Using the symplectic cut and symplectic sum constructions,
we express the open Gromov-Witten invariants of $(X,L)$ in terms
of open Gromov-Witten invariants of a pair $(X_-,L)$ determined by $L$
and the standard Gromov-Witten invariants of a symplectic manifold $X_+$
determined by $(X,L)$. We also describe other applications of this
approach.
\end{abstract}
\maketitle
\tableofcontents
\section{Introduction\label{ch:introduction}}

Let $(X,\om,\phi)$ be a symplectic manifold, which we will assume to be connected throughout this paper, 
with a real structure $\phi$, i.e a diffeomorphism $\phi\colon X\to X$ such that $\phi^2=\id_X$ and  
$\phi^*\om=-\om$. Let $L=\Fix(\phi) \subset X$ be the fixed point locus of $\phi$; $L$ is a Lagrangian submanifold of $(X,\om)$ which can be empty. An almost complex structure $J$ on $TX$ is called \textsf{$(\om,\phi)$-compatible} if $\phi^*J=-J$ and $\om(\cdot,J\cdot)$ is a metric.

Fix a compatible almost complex structure $J$. We define $\mc{M}_{k,l}^{\disc}(X,L,\beta)$ to be the moduli space of somewhere injective $J$-holomorphic discs
\begin{equation}\label{J-disc}
u\colon (D^2,S^1) \to (X,L),
\end{equation}
$$ du+J\circ du\circ j=0,\qquad u^{-1}(u(z))=\{z\} \quad \textnormal{for almost every }z\in \P^1,$$
\noindent with $k$ boundary marked points and $l$ interior marked points, with $k+2l=n$, representing some relative homology class $\beta\in H_2(X,L)$. By \cite[Appendix C]{MS2}, this moduli space has real virtual dimension
$$\dim^{\vir}\mc{M}_{k,l}^{\disc}(X,\beta)= \dim_{\C}X+\mu(\beta)+n-3,$$ 
where $\mu(\beta)$ is the Maslov index of $\beta$. If $L=\Fix(\phi)$, corresponding to every map (\ref{J-disc}) there is another $J$-holomorphic map
$$\tilde{u}=\tau_{\mc{M}}(u)=\phi\circ u \circ c,$$
where $c(z)=\bar{z}$ is the complex conjugation on $D^2$. The two discs above need not be in the same homology class, but both have the same Maslov index and symplectic area; see \cite[Definition 2.4.17]{FOOO}. In the presence of an involution, we define an equivalence relation $\sim$ on $H_2(X,L)$ by
\begin{equation}\label{equivalence-relation}
\beta_1\sim\beta_2 \;\Leftrightarrow \; \mu(\beta_1)=\mu(\beta_2),\; \omega(\beta_1)=\omega(\beta_2).
\end{equation}
In the the definition of $\mc{M}_{k,l}^{\disc}(X,L,\beta)$, we often consider $\beta$ to be an element of $H_2(X,L)/\sim$, instead of $H_2(X,L)$.\\

Let $\ov{\mc{M}}_{k,l}^{\disc}(X,L,\beta)$ be the stable map compactification of  $\mc{M}_{k,l}^{\disc}(X,L,\beta)$; see \cite[Section~7]{FOOO} for the definition. Let 
\begin{equation}\label{evaluation-maps}
\begin{aligned}
& \ev_i^{B}\colon \ov{\mc{M}}_{k,l}^{\disc}(X,L,\beta) \to L,  &\ev_i^{B}([u,\Si,(w_j)_{j=1}^k,(z_j)_{j=1}^{l})])= u(w_i),\\
& \ev_i\colon \ov{\mc{M}}_{k,l}^{\disc}(X,L,\beta) \to X,     &\ev_i([u,\Si,(w_j)_{j=1}^k,(z_j)_{j=1}^{l})])= u(z_i),
\end{aligned}
\end{equation}
be the natural evaluation maps.

For the classic moduli space $\ov{\mc{M}}_{n}(X,A)$ of $J$-holomorphic spheres in homology class $A$, \textsf{Gromov-Witten invariants} are defined via integrals of the form
\begin{equation}\label{GW-invariants}
\left\langle \theta_1,\cdots,\theta_n\right\rangle_{A}= \int_{[\ov{\mc{M}}_{n}(X,A)]^{\vir}} \ev_1^*(\theta_1)\wedge \cdots \wedge \ev_n^*(\theta_n),
\end{equation}
where $\theta_i$'s are cohomology classes on $X$ (see \cite{FO},\cite{LT},\cite{RT}). These integrals make sense and are independent of $J$ as $\ov{\mc{M}}_{n}(X,A)$ has a virtually orientable fundamental cycle without real codimension one boundary.
The existence of similar invariants for the moduli spaces and evaluation maps in (\ref{evaluation-maps}) is predicted by physicists (\cite{F1}, \cite{F2}, \cite{F3}, \cite{F4}), but there are obstacles to defining such invariants mathematically. In addition to the transversality issues (which are also present in the classical case), issues concerning orientability and codimension one boundary arise.

Whereas moduli spaces of closed curves have a canonical orientation induced by $J$, $\ov{\mc{M}}_{k,l}^{\disc}(X,L,\beta)$ is not necessarily orientable. Moreover, if it is orientable, there is no canonical orientation. If $L$ has a spin structure, then $\ov{\mc{M}}_{k,l}^{\disc}(X,L,\beta)$ is orientable and a choice of spin structure canonically  determines an orientation on $\ov{\mc{M}}_{k,l}^{\disc}(X,L,\beta)$; see \cite[Section 8]{FOOO}.

The moduli spaces $\ov{\mc{M}}_{k,l}^{\disc}(X,L,\beta)$ have two types of codimension one boundary; see Figure \ref{Fig:boundaryterms}. 
The first type, called \textsf{disc bubbling}, consists of maps from two discs with a boundary point in common. This boundary breaks into unions of components isomorphic to 
\begin{equation}\label{boundary-equ1}
\mc{M}_{1+k_1,l_1}^{\disc}(X,L,\beta_1) \times_{(\ev^B_1,\ev^B_1)} \mc{M}_{1+k_2,l_2}^{\disc}(X,L,\beta_2)/G,
\end{equation}
where 
$$k_1+k_2=k,\quad l_1+l_2=l,\quad \beta_1+\beta_2=\beta, \quad
G=\begin{cases}\Z_2,& \hbox{if}~k,l=0,~\beta_1=\beta_2;\\
\{1\},& \hbox{otherwise}.
\end{cases}$$
The second type, called \textsf{sphere bubbling}, appears only if $k=0$ and $\beta$ lies in the image of the natural homomorphism 
$j\colon H_2(X)\to H_2(X,L)$.
It consists of maps from $\P^1$ taking an extra marked point to $L$. This boundary is isomorphic to 
\begin{equation}\label{boundary-equ2}
 \bigsqcup_{\tilde\beta\in j^{-1}(\beta)} \mc{M}_{1+l}(X,\tilde{\beta})\times_{\ev_1} L.
\end{equation}
The boundary problem is present in nearly all cases. It has been overcome in a number of cases by either adding other terms to compensate for the effect of the boundary (\cite{W1}, \cite{W2}, \cite{F}) or by gluing boundary components to each other to get moduli spaces without boundary (\cite{S}, \cite{G}). None of these methods can address the issue of sphere bubbling; we overcome this issue in \cite{MF2}, by adding the contribution of real curves without fixed points.
\begin{figure}
\includegraphics[scale=.4]{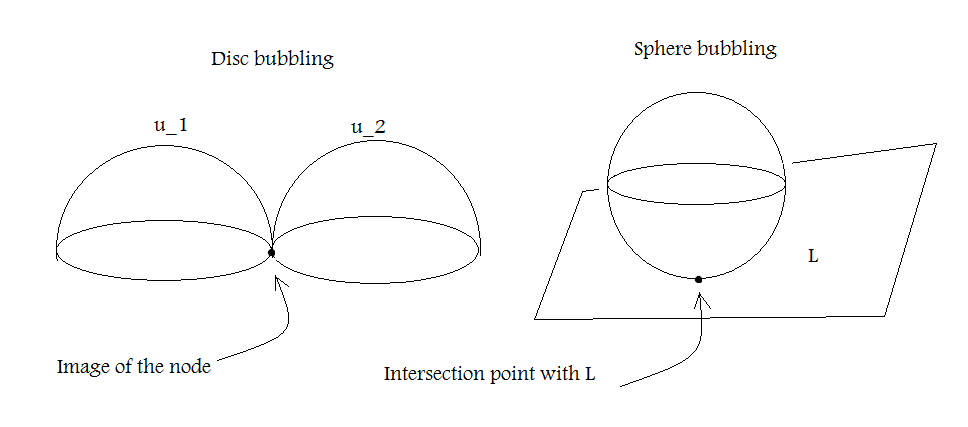}
\caption{The codimension one boundary in $\ov{\mc{M}}_{k,l}^{\disc}(X,L,\beta)$}\label{Fig:boundaryterms}
\end{figure}  

In \cite{S}, the disc bubbling problem is resolved by using the involution $\phi$ on $X$ to identify the disc bubbling boundary components and thus define a moduli space without such boundary. If the sphere bubbling does not happen, e.g.~when $\partial\beta \neq 0 \in H_1(X,L)$, the resulting moduli space is orientable and gives rise to invariants of $(X,\om,\phi)$. 

Invariants arising from the moduli spaces and evaluation maps (\ref{evaluation-maps}) are called \textsf{open} GW invariants. Such invariants have been defined in a number of settings by Liu \cite{Liu}, Welschinger \cite{W3},\cite{W4},\cite{W5}, Solomon \cite{S}, Fukaya \cite{F}, and  Georgieva \cite{G}.
Open GW counts for $(X,L)$, when $X$ is a Calabi-Yau threefold and $L\cong S^3$, are defined in \cite{F}.
These counts depend on the choice of almost complex structure via a wall-crossing formula, due to the sphere bubbling issue.
They are a priori real numbers, which are predicted  in \cite{F} and shown in this paper
to be rational; see \cite[Conjecture 8.1]{F} and Corollary~\ref{coro:rationality} below.

In this paper, we use degeneration techniques to get a better understanding of moduli spaces of $J$-holomorphic discs and to compute open GW invariants. If the Lagrangian $L$ is $S^n$, $\R\P^n$, or a Lens space of a certain type, there is a canonical Hamiltonian $S^1$-action in a Weinstein neighborhood of $L$. We then, via the symplectic cut corresponding to this $S^1$-action, degenerate the symplectic manifold $X$ to a nodal singular symplectic manifold, $X_+\cup_{D} X_{-}$, where $D$ is the intersection divisor, through a family of symplectic manifolds over a small disk, $\pi \colon\mc{X}\to\De$. Through this degeneration process $L$ goes into $X_-$ and the symplectic type of $X_-$ depends only on $L$. We then relate the moduli spaces in the smooth fibers to some fiber product of the moduli spaces in the singular fiber. As an application, we obtain the following result.

\begin{theorem}\label{emptiness}
Let $(X^{2n},\omega)$ be a symplectic manifold with $n\ge 3$ and $c_1(TX)=0$. If $L \subset X$ is a  Lagrangian submanifold 
diffeomorphic to $S^n$ and $E\in\R^+$, there is an open subset $U_E$ of the set $\mc{J}$ of all 
$(\om,L)$-compatible almost complex  structures such that the moduli space $\ov{\mc{M}}^{\disc}(X,L,J,\beta)$
is empty whenever $\om(\beta)<E$ and $J\in U_E $.
\end{theorem}

This result is also stated in \cite[Corollary 4.3]{W6}.
Its proof in \cite{W6} involves degenerating the almost complex structure to a singular one obtained by stretching 
a neighborhood of the Lagrangian and studying the behavior of the moduli space in the limit. 
This stretching surgery appears in Symplectic Field Theory \cite{SFT} and can be performed near any Lagrangian manifold in any symplectic manifold, 
but the result is a non-compact manifold.  
By contrast, our techniques work only if there is a Hamiltonian $S^1$-action in $T^*L$,
but we get closed symplectic manifolds in the end.
We show that as we move toward certain exotic almost complex structures, the sphere bubbling happens for all $J$-holomorphic 
discs and they all disappear.

\begin{corollary}\label{coro:rationality}
The disc counts for $(X,L)$, where $(X,\om)$ is a Calabi-Yau threefold and $L \subset X$ is a  Lagrangian submanifold diffeomorphic to $S^3$,
defined by \cite{F} are rational numbers. 
\end{corollary} 

This corollary, which confirms \cite[Conjecture 8.1]{F}, follows immediately from Theorem~\ref{emptiness}, since
the wall crossing changes the disc counts defined in \cite{F} by rational numbers; see \cite[Section 6]{F}. 

\begin{corollary}\label{coro:nondispla}
Let $(X^{2n},\omega)$ be a symplectic manifold with $n\ge 3$ and $c_1(TX)=0$. If $L \subset X$ is a  Lagrangian submanifold 
diffeomorphic to $S^n$, then $L$ is not displacable, i.e.~there exists no Hamiltonian isotopy $\psi_t\!:X\to X$ such that 
$\psi_1(X)\cap X=\emptyset$.
\end{corollary} 

This corollary is a special case of \cite[Theorem~H]{FOOO}, but our argument avoids the technical issues that are the focus
of~\cite{FOOO}.
Since there are no $J$-holomorphic discs in $(X,L)$, there is no difficulty in defining the Floer homology groups
for Lagrangians $L_1,L_2$, with either $L_1=L_2$ or $L_1\pitchfork L_2$, as described in \cite[Section 1.1]{FOOO},
or showing that they are preserved when either Lagrangian is deformed by a Hamiltonian isotopy. 
Thus, if $\psi_t\!:X\to X$ is any Hamiltonian isotopy such that $L\pitchfork \psi_1(L)$, then
$$HF^*(L,\psi_1(L)) \cong HF^*(L,L) \cong H^*(L),$$
which implies that $L$ is not displacable.

Let $\pi\colon\mc{X}\to \De$ be a family of symplectic manifolds over a disk $\De \subset \C$ obtained from the symplectic sum construction for $X_+\cup_D X_-$; see the paragraph preceding Theorem~\ref{emptiness}. If $L$ is $S^n$, $\R\P^n$, or a Lens space of a certain type and its Weinstein neighborhood used in the symplectic cut process is chosen appropriately, an antisymplectic involution~$\phi$ on $X$ such that $L=\Fix(\phi)$ induces an antisymplectic involution $\phi_{\mc{X}}$ on $\mc{X}$ covering the standard conjugation on $\De$ such that $\Fix(\phi_{\mc{X}})\cap X_+=\emptyset$; see Corollary~\ref{coro:main-coro}.
Furthermore, $(X_{-},\phi_{-})$ is a symplectic manifold with a real structure independent of $X$. For example, if $L\cong \R\P^n$, then $(X_{-},\om_-,\phi_{-})$ is symplectomorphic to $(\P^n,\omega_{\tn{FS}},\tau_n)$, where $\omega_{FS}$ is some multiple of the Fubini-Study form and $\tau_n$ is the standard complex conjugation; if $L\cong S^n$, then $(X_{-},\om_-\phi_{-})$ is symplectomorphic to $(Q^n,\omega_{FS},\tau_{n+1})$, where $Q^n\subset \P^{n+1}$ is a quadratic hypersurface given by a real equation.

By contrast with the spherical case of Theorem \ref{emptiness}, non-trivial open GW invariants do exist when $L$ is diffeomorphic to the real projective plane. For example, the odd-degree open invariants of the quintic threefold computed in \cite{PSW} are not zero.
In \cite{PSW}, equivariant localization is used to reduce the computation in the open case to the closed case. 
Our approach is different, but similar in flavor: we use degeneration to reduce many computations in the open case to the closed case. 

\begin{theorem}\label{open-closed}
Let $(X^6,\omega,\phi)$ be a symplectic manifold with $c_1(TX)=0$ and $L \cong \R\P^3$. 
For any equivalence class $\beta \in H_2(X,L)/\sim$ with $\partial{\beta}\neq 0 \in H^1(L)$, 
the open GW invariants $N_{\beta}^{\disc}$ are a universal linear combination of the classical GW invariants of a symplectic 6-fold $X_+$ with $c_1(TX_+)=0$ canonically constructed from $X$.
\end{theorem}

This is proved in Section \ref{proofs}, where we derive an explicit formula relating the open GW invariants of $X$ and the classical invariants of $X_+$. Although we state these theorems for Calabi-Yau manifolds, our degeneration technique can be used for any manifold $X$, 
as long as $L$ is $S^n$, $\R\P^n$, or some other special Lens space; see Section~\ref{ch:surgery}. 

We now give a detailed version of the statement of Theorem 1.4. The manifold $X_+$ is obtained from $X$ by replacing a tubular
neighborhood~$U$ of $L$, which in this case is symplectomorphic~to a neighborhood of $L$ in
$$TL\cong \R\P^3\!\times\!\R^3\cong\P^3\!-\!Q^2,$$
with a tubular neighborhood~$U_+$ of $Q^2$ in the line bundle ${\mc{O}}_{\P^3}(-2)|_{Q^2}$. 
In particular, we replace the Lagrangian $L$ in $X$ by a symplectic divisor $D\cong Q^2$ in~$X_+$. 
Since $X_+$ is a Calabi-Yau threefold, the virtual dimension of $\ov{\mc{M}}_0(X_+,\beta_+)$ is~0 for every $\beta_+\!\in\!H_2(X_+)$;
we denote the corresponding GW-invariant, i.e.~the virtual degree of $\ov{\mc{M}}_0(X_+,\beta_+)$ by $N_{\beta_+}(X_+)$.
The exact sequence for the pair $(X_+,\bar{U}_+)$, excision, and homotopy give rise to a homomorphism
$$j_+\!:H_2(X_+)\ra H_2(X_+,\bar{U}_+)\cong H_2(X_+\!\setminus\!U_+,\partial\bar{U}_+) =H_2(X\!\setminus\!U,\partial\bar{U})\cong H_2(X,\bar{U})\cong H_2(X,L).$$
For each $\beta\!\in\!H_2(X,L)$ and $d\!\in\!\Z$, let
$$N_{\beta,d}(X_+)=\sum_{\begin{subarray}{c}\beta_+\in H_2(X_+)\\ j_+(\beta_+)=\beta,~\beta_+\cdot D=d\end{subarray}}\hspace{-.3in} N_{\beta_+}(X_+).$$
Similarly, we define
\begin{equation}\label{coeffnums_e}
\begin{split}
N_d(\P^3)&=\int_{\ov{\mc{M}}_d(\P^3,(d/2)\ell)}\ev_1^*(\tn{PD}_{\P^3}\tn{pt})\ldots\ev_d^*(\tn{PD}_{\P^3}\tn{pt}),\\
N_d^{\disc}(\P^3,\R\P^3)&=\int_{\ov{\mc{M}}_d^{\disc}(\P^3,\R\P^3,d(\ell/2))}\ev_1^*(\tn{PD}_{\P^3}\tn{pt})\ldots\ev_d^*(\tn{PD}_{\P^3}\tn{pt}),
\end{split}
\end{equation}
where $\ell\in H_2(\P^3)$ and $\ell/2\!\in\!H_2(\P^3,\R\P^3)$ are the standard generators
and $N_d(\P^3)\equiv0$ if $d$ is odd.
The numbers~(\ref{coeffnums_e}) can be computed using equivariant localization; the first set of numbers can also be computed via the recursion of \cite[Theorem~10.4]{RT}, and a similar recursion for the second set of numbers has been announced~\cite{So2}.
The formula of Theorem 1.4 is of the form
\begin{equation*}
N_{\beta}^{\disc}(X,L)=\sum_{|\Gamma|=\beta}
\frac{1}{\mf{Aut}(\Gamma)}N_{\Gamma}^{\disc}(\P^3,\R\P^3)N_{\Gamma}(X_+),
\end{equation*}
where the sum is over labeled bipartite rooted graphs $\Gamma\!=\!(\Gamma_-,\Gamma_+)$ such~that 
\begin{enumerate}[label=(\arabic*),leftmargin=*]
\item $\Gamma$ is connected and has no loops,
\item the root (distinguished vertex) $v_0$ is an element of the set $\Gamma_-$, and
\item each vertex of $\Gamma_+$ is decorated by an element $\beta_v$ of $H_2(X,L)$,
with all these labels  summing up to~$\beta$.
\end{enumerate}
The root $v_0$ corresponds to a $J$-holomorphic disc in $\P^3$ with boundary in $\R\P^3$, while each of the remaining vertices of~$\Gamma$ corresponds to a $J$-holomorphic sphere in $X_-$ or $X_+$, depending on whether the vertex lies in~$\Gamma_-$ or~$\Gamma_+$; see Figure~\ref{bipartite}.

\begin{figure}
\includegraphics[scale=.6]{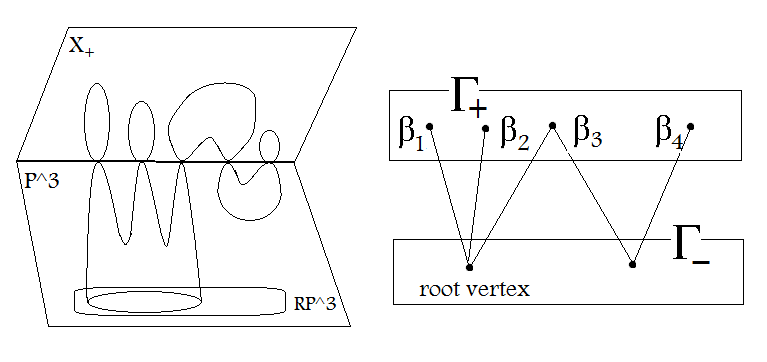}
\caption{A possible configuration of curves in central fiber and the corresponding $\Gamma$.}\label{bipartite}
\end{figure}

Each edge of~$\Gamma$ corresponds to an intersection point with $D\!\subset\!\P^3,X_+$ common to the components corresponding to the two vertices of the edge; the order of the intersection is~one. For a graph~$\Gamma$ as above  any vertex~$v$ of~$\Gamma$, let $\tn{E}_v(\Gamma)$ be the number of edges
of~$\Gamma$ leaving~$v$.
We denote by $\mf{Aut}(\Gamma)$ the order of the automorphism group of~$\Gamma$ and define
$$N_{\Gamma}^{\disc}(\P^3,\R\P^3)=N_{\tn{E}_{v_0}(\Gamma)}^{\disc}(\P^3,\R\P^3)\cdot \prod_{v\in\Gamma_-\setminus v_0}\!\!\!\!\!N_{\tn{E}_v(\Gamma)}(\P^3),\quad N_{\Gamma}(X_+)=\prod_{v\in\Gamma_+}N_{\beta_v,\tn{E}_v(\Gamma)}(X_+).$$

In Section \ref{ch:surgery}, we describe our degeneration setting, reviewing the symplectic cut and sum constructions along the way. We also show that every antisymplectic involution $\phi$ is ``standard'' in a properly chosen Weinstein neighborhood of $\Fix(\phi)$; see Lemma~\ref{deforming-involution0}.
In Section \ref{ch:holomorphicdiscs}, we first review
the definition of open GW invariants, then discuss some examples for
Calabi-Yau threefolds, and finally introduce the notion of relative open invariants.
In Section \ref{ch:degeneration}, we construct moduli spaces of discs over a family of symplectic
manifolds obtained from the symplectic sum procedure, paying special attention to the moduli space over singular fiber, and then prove Theorems \ref{emptiness} and \ref{open-closed}.\\

I would like to thank my advisor, Gang Tian, for his continuous encouragement and support. 
I also wish to thank Aleksey Zinger for his help with the exposition.

\section{A fibration corresponding to Lens spaces\label{ch:surgery}}

The symplectic cut procedure, introduced in \cite{L}, is a surgery technique for symplectic manifolds by means of which  we can decompose a given symplectic manifold into two pieces, each again a symplectic manifold. There is an inverse operation, the symplectic sum, that glues two manifolds into one.

Let $\De\subset \C$ denote a disk centered at the origin and $\De^*=\De\setminus 0$. If $\pi\colon \mc{X} \to \De$ is any map and $\la\in \De$, let $\mc{X}_{\la}\equiv \pi^{-1}(\la)$ be the fiber over $\la$. A \textsf{symplectic fibration} is a pair ($\pi\colon \mc{X} \to \De,\om_{\mc{X}}$) such that
$\pi$ is surjective, $(\mc{X},\om_{\mc{X}})$ is a symplectic manifold, $\mc{X} _{\la}$ is a symplectic submanifold of $(\mc{X},\om_{\mc{X}})$ for every $\la \in \De^*$, and $\mc{X}_{0}$ is a union of symplectic submanifolds of $(\mc{X},\om_{\mc{X}})$ meeting along smooth symplectic divisors. A \textsf{Lagrangian subfibration} of $(\pi\colon \mc{X} \to \De, \om_{\mc{X}})$ is a submanifold $\mc{L}\subset\mc{X}$ disjoint from the singular locus of~$\mc{X}_0$ such that $\pi(\mc{L})=\De$ and $\mc{L}_{\la}\subset\mc{X}_{\la}$ is a Lagrangian submanifold for every $\la\in\De$. Thus, $\mc{L}\cong \mc{L}_0\times\De$ as fibrations over $\De$ and $(\mc{X}_{\la},\mc{L}_{\la})$ is symplectically isotopic to $(\mc{X}_{\la'},\mc{L}_{\la'})$ for all $\la,\la'\in\De^*$.

An \textsf{admissible almost complex structure} on a symplectic fibration $(\pi\!:\mc{X}\to\De,\om_{\mc{X}})$ is an $\om_{\mc{X}}$-compatible almost complex structure on $\mc{X}$ which preserves $\ker d\pi$, restricts to an almost complex structure on the singular locus~$D$ of~$\mc{X}_0$, and satisfies
$$N_{J_{\mc{X}}}(u,v)\in T_xD \qquad\forall~u\in T_xD,~v\in T_x\mc{X}_0,~x\in D,$$
where $N_{J_{\mc{X}}}$ is the Nijenhuis tensor of  $J_{\mc{X}}$. We denote the set of all admissible almost complex structures on $\mc{X}$ by~$\mc{J}_{\mc{X}}$. A \textsf{real structure} on a symplectic fibration $(\pi\!:\mc{X}\to\De,\om_{\mc{X}})$ is an anti-symplectic involution $\phi_{\mc{X}}\!:\mc{X}\to\mc{X}$ covering the standard complex conjugation on~$\De$. A \textsf{$\phi_{\mc{X}}$-compatible admissible almost complex structure} on $(\pi\!:\mc{X}\to\De,\om_{\mc{X}})$ is an element $J_{\mc{X}}\in\mc{J}_{\mc{X}}$ such that $\phi_{\mc{X}}^*J_{\mc{X}}=-J_{\mc{X}}$.
We denote the set of such almost complex structures by $\mc{J}_{\phi_{\mc{X}}}$.

A \textsf{Lens space} is the quotient of $S^n$ by a free $\Z_k$-action by isometries. Every fixed-point free map $S^n\to S^n$ is homotopic to the antipodal map, which has degree $-1$ if $n$ is even. Thus, if $\Z_k$ acts freely on $S^n$ with $n$ even, then $k\le2$. If $n=2m\!-\!1$, a free action of $\Z_k$ on $S^n\subset\C^m$ by isometries is generated~by a map of the form
$$S^n\to S^n, \qquad (z_1,\ldots,z_m)\ra(\xi_1z_1,\ldots,\xi_m z_m),$$
for some primitive $k$-th roots $\xi_1,\ldots,\xi_m$ of 1. Any such action extends to an action~on 
\fontsize{10}{12}\selectfont
\begin{gather*}
Q^n\equiv\big\{[z_0,\ldots,z_{n+1}]\!\in\!\P^{n+1}\!:~z_0^2=\sum_{j=1}^{n+1}z_j^2\big\}
\supset Q^n_{\R}\equiv Q^n\cap \R\P^n \cong S^n,\\
[z_0,\ldots,z_{n+1}]\to
\begin{cases} [z_0,-z_1,\ldots,-z_{n+1}],\!\!\!&\hbox{if}\,k\!=\!2,\\
[z_0,\xi_1^{\R}z_1\!+\!\xi_1^{\mf{i}\R}z_2,\xi_1^{\R}z_2\!-\!\xi_1^{\mf{i}\R}z_1,\ldots,
\xi_m^{\R}z_n\!+\!\xi_m^{\mf{i}\R}z_{n+1},\xi_m^{\R}z_{n+1}\!-\!\xi_m^{\mf{i}\R}z_n],\!\!\!&
\hbox{if}\,n\!=\!2m\!-\!1,
\end{cases}
\end{gather*}
\normalsize
where $\xi_j^{\R}$ and $\xi_j^{i\R}$ are the real and imaginary parts of $\xi_j$, respectively.
This extension preserves the divisor $$D_n\equiv \big\{[z_0,\ldots,z_{n+1}]\!\in\!Q^n\!:~z_0\!=\!0\big\}\cong Q^{n-1}.$$
We call a Lens space $S^n/\Z_k$ \textsf{archetypal} if $k=1,2$ or if the induced action of $\Z_k$ on $Q^n$ is free.  
Furthermore, $Q^n_{\R}$ is the fixed-point locus of the restriction $\tau_n^Q$ to $Q^n$ of the standard involution 
$$\tau_{n+1}\!:\P^{n+1}\ra\P^{n+1},\qquad [z_0,\ldots,z_{n+1}]\ra[\bar{z}_0,\ldots,\bar{z}_{n+1}].$$
This restriction induces an anti-symplectic involution on the quotient $Q^n/\Z_k$, which we still denote by~$\tau_n^Q$. In this section,  we establish the following proposition, which is used in the reminder of the paper to study moduli spaces of $J$-holomorphic discs.

\begin{proposition}\label{prop:main-space}
Let $(X,\om)$ be a symplectic $n$-fold with a Lagrangian $L$ diffeomorphic to an archetypal Lens space $S^n/\Z_k$. 
There exists a symplectic fibration $\pi\!:\mc{X}\!\to\!\De$ with a Lagrangian subfibration~$\mc{L}$
and $\mc{X}_0=X_-\!\cup_D\!X_+$, where $X_-$ and $X_+$ are symplectic manifolds and $D=X_-\cap X_+$, so~that 
\begin{enumerate}[leftmargin=*]
\item for $\la\in\De^*$, $(\mc{X}_{\la},\om_{\mc{X}}|_{X_{\la}},\mc{L}_{\la})$ is symplectically isotopic to $(X,\om,L)$;
\item $\mc{L}_0\subset X_-$ and $(X_-,\om_{\mc{X}}|_{X_-},D,\mc{L}_0)$ is symplectomorphic to $$(Q^n/\Z_k,\om_{\FS},D_n/\Z_k,Q^n_{\R}/\Z_k);$$
\item $c_1(TX_+)=\frac{2k-\de_{k,2}-n}{k}\tn{PD}_{X_+}[D]$ if $c_1(TX)=0$.
\end{enumerate}
Moreover, the space $\mc{J}_{\mc{X}}$ is non-empty and path-connected. 

If in addition $L=\Fix(\phi)$ for a real structure $\phi$ on $(X,\om)$, the above symplectic fibration can be chosen so that it admits a real structure $\phi_{\mc{X}}$ with $\mc{L}\!=\!\Fix(\phi_{\mc{X}})$, $(X_{\la},\om_{\mc{X}}|_{X_{\la}}, \phi|_{X_{\la}})$ symplectically isotopic to $(X,\om,\phi)$ for every $\la\in\De^*$, and $(X_-,\om_{\mc{X}}|_{X_-},D,\phi_{\mc{X}}|_{X_-})$
symplectomorphic to $(Q^n/\Z_k,\om_{\FS},D_n/\Z_k,\tau_n^Q)$. In this case, the space $\mc{J}_{\phi_{\mc{X}}}$ is  non-empty and path-connected.
\end{proposition}

\begin{remark}
If $k=2$, then $L=\R\P^n$ and $(X_-,\om_-)\cong(\P^n, \om_{\FS})$. Moreover, if $n=3$, then $c_1(TX_+)=0$.
\end{remark}

\begin{remark}
If the action of $\Z_k$ on $Q^n$ is not free, then $Q^n/\Z_k$ may be an orbifold and we get a fibration $\mc{X}$ which is singular along $D$. By resolving the singularities, we can get a fibration similar to that of Proposition~\ref{prop:main-space}, but in this paper we are only interested in the $k=1,2$ cases.
\end{remark}

\begin{remark}                            
Let $\mc{SM}$ be the category of all projective varieties (of any dimension). 
Let $\mc{SM}^+$ be the free abelian group generated by $\mc{SM}$ and  $\mc{RS} \subset \mc{SM}^+ $ be the set of all double point relations
$$ [X]-[X_{-}]-[X_{+}]+[\P(\mc{N}^{X_{\pm}}_{D}\oplus \C)], $$
where $X$ is the symplectic sum of $(X_{\pm},D)$ and $\P(\mc{N}^{X_{\pm}}_{D}\oplus \C)$ is a $\P^1$ bundle over $D$. By \cite[Corollary 3]{LP}, $\mc{SM}^+ / \mc{RS}$ is generated by the product of projective spaces. Thus, theoretically a problem can be reduced to one for projective spaces if we know how things change in a symplectic sum/cut.
\end{remark}

In Section~\ref{sec:cut}, we review the symplectic cut procedure and apply it to a canonical $S^1$-action on $T^*S^n$. In Section~\ref{sec:sum}, we review the symplectic sum procedure, apply it to the result of the symplectic cut of the previous section, and build the symplectic fibration of Proposition~\ref{prop:main-space}. Finally, in Section~\ref{sec:involution}, we show that a real structure on the starting manifold induces an involution on the symplectic fibration of Proposition~\ref{prop:main-space}.


\subsection{Symplectic cut}\label{sec:cut}

Let $(X^{2n},\om)$ be a symplectic manifold. Let $U\subset X$ be an open subset with a Hamiltonian $S^1$-action
$$ S^1 \times U \to U,\qquad (e^{\mf{i}\theta},p)\to e^{\mf{i}\theta}p,$$
with moment map $h\colon U\to \R$ and fundamental vector field $\ze_h\in\Gamma(U;TU)$ so that 
$$\ze_h(p)=\frac{\td}{\td\theta}\left(e^{\mf{i}\theta}p\right)\!|_{\theta=0}\quad \forall p\in U, \qquad dh=\iota_{\ze_h}\om \equiv \om(\ze_h,\cdot).$$
For any $a\!\in\!\R$, let $V_a\!=\!h^{-1}(a)$.
The $S^1$-action on $U$ restricts to an action on~$V_a$ and induces an action~on 
$$V_a(\ep)\equiv \big\{(p,t)\!\in\!V_a\!\times\!\R\!:~|t|\!<\!\ep(p)\big\}\subset V_a\!\times\!\R$$
for every $S^1$-invariant smooth map  $\ep\colon V_a\to\R^+$,
by acting trivially on the second factor.

\begin{lemma}\label{tubular_lmm}
If $a\!\in\!\R$ is a regular value for $h$, there exist an $S^1$-invariant smooth map 
$\ep\!:V_a\!\to\R^+$, an $S^1$-invariant neighborhood $U'$ of $V_a$ in~$U$, and an $S^1$-equivariant diffeomorphism 
$$\Psi\!:V_a(\ep)\to U' \qquad\tn{s.t.}\quad \Psi|_{V_a\times0}=\pi_1,~~~h\circ\Psi=a+\pi_2,$$
where $\pi_1,\pi_2\!:V_a\!\times\!\R\!\to\!V_a,\R$ are the two projections.
If $V_a$ is compact, $\ep$ can be taken to be a constant $\ep$ and 
$U'=h^{-1}((a\!-\!\ep,a\!+\!\ep))$.
\end{lemma}

\begin{proof}
Let $J$ be an $\om$-compatible $S^1$-invariant almost complex structure and 
$$g_J(\cdot,\cdot)\equiv\om(\cdot,J\cdot)$$ be the corresponding metric.\footnote{Such a $J$ can be 
obtained by first averaging any $\om$-compatible metric $g$ over the $S^1$-action.}
By shrinking~$U$ if necessary, we can assume that $\ze_h$ does not vanish on~$U$.
Let $\Phi\!:W\!\to\!U$ be the flow of the vector field
$$\xi_h=\frac{1}{g_J(\ze_h,\ze_h)}J\ze_h\in\Gamma(U;TU);$$
see \cite[Theorem~1.48]{W}. In particular, $W$ is an open neighborhood of $U\!\times\!0$ in $U\!\times\!\R$.
Since $\xi_h$ is $S^1$-equivariant, $W\!\subset\!U\!\times\!\R$ is preserved 
by the $S^1$-action on $U\!\times\!\R$ and $\Phi$ is $S^1$-equivariant.
Since
$$\frac{\td}{\td t}h\big(\Phi(x,t)\big) =\td_{\Phi(x,t)}h\big(\xi_h(\Phi(x,t))\big)
=\om\big(\ze_h(\Phi(x,t)),\xi_h(\Phi(x,t))\big)=1,$$
$h(\Phi(x,t))=h(x)+t$.
Let
$$W_a=\big\{(x,t)\!\in\!W\!:\,x\!\in\!V_a\big\}.$$
Since 
$$\frac{\td}{\td t}\Phi(x,t)\Big|_{t=0}\not\in \ker\td_xh=T_xV_a\qquad\forall\,x\!\in\!V_a,$$
the differential
$$\td_{(x,0)}\Phi|_{W_a}\!: T_xV_a\!\oplus\!\R \to T_xU$$ 
is an isomorphism for all $x\!\in\!V_a$. 
Thus, by the Inverse Function Theorem, there exists a smooth function
$\ep\!:V_a\!\to\!\R^+$, which can be chosen to be $S^1$-invariant, so that
$\Phi|_{W_a}$ restricts to a diffeomorphism from $V_a(\ep)$ onto its image~$U'$.
\end{proof}

\noindent
Extending the Hamiltonian $S^1$-action on $U$ to $(\tilde{U}\!\equiv\!U\!\times\!\C,\om\!\oplus\!\om_0)$
by $e^{\pm\mf{i}\theta}$ in the second factor, let
$$h_{\pm}\!:U\!\times\!\C\to\R, \qquad (x,z)\to h(x)\mp\frac{1}{2}|z|^2,$$
denote its moment map and define 
$$V_{\pm;a}=h_{\pm}^{-1}(a), \qquad D_a=(V_{\pm;a}\cap U\!\times\!0)/S^1\subset V_{\pm;a}/S^1.$$

\begin{lemma}
If $a\!\in\!\R$ is a regular value for $h$ and $S^1$ acts freely on $V_a$,
$D_a$ is a symplectic submanifold of  $V_{\pm;a}/S^1$ with symplectic normal
bundle isomorphic~to
$$L_{\pm}=V_a\times_{S^1}\C, \qquad (x,z)\sim\big(e^{\mf{i}\theta}x,e^{\pm\mf{i}\theta}z\big).$$
\end{lemma}

\begin{proof} By \cite[Theorem~1]{MW}, there are unique symplectic forms $\om_{\pm}$
and $\om_{\pm}'$ on $V_{\pm;a}/S^1$ and $D_a$, respectively, characterized by
$$\pi_{\pm}^*\om_{\pm}=(\om\!\oplus\!\om_0)\big|_{TV_{\pm;a}}, \qquad
\pi_{\pm}'^*\om_{\pm}'=(\om\!\oplus\!0)|_{TV_{a}}=(\om\oplus\om_0)\big|_{TV_{a}},$$
where $\pi_{\pm}\!:V_{\pm;a}\to V_{\pm;a}/S^1$ and
$\pi_{\pm}'\!:V_{\pm;a}\cap U\!\times\!0\to D_a$ are the quotient maps.
Thus, $D_a$ is a symplectic submanifold of $V_{\pm;a}/S^1$.
With $\ep$ and $\Psi$ as in Lemma~\ref{tubular_lmm}, the map
$$\widetilde\Psi\!:\big\{[x,z]\!\in\!L_{\pm}\!:\,|z|^2\!<\!2\ep(x)\big\}\to \big(V_{\pm;a}\cap V_a(\ep)\!\times\!\C\big)/S^1,$$
$$ [x,z]\to \big[\Psi(x,\pm|z|^2/2),z\big]$$
is a diffeomorphism such that 
$$\{\widetilde\Psi^*\om_{\pm}\}(w_1,w_2)=\om_0(w_1,w_2)\qquad\forall~w_1,w_2\in\C\subset T_xL_{\pm},~x\in D_a.$$
This implies the claim.
\end{proof}

\noindent
Extending the Hamiltonian $S^1$-action on $U$ to $(\tilde{U}\!\equiv\!U\!\times\!\C\!\times\!,\om\!\oplus\!\om_0\!\oplus\!\om)$
by $e^{\mf{i}\theta}$ in the second factor and $e^{-\mf{i}\theta}$ in the third factor, 
let
$$\tilde{h}\!:U\!\times\!\C\!\times\!\C\to\R, \qquad (x,z_+,z_-)\to h(x)-\frac{1}{2}|z_+|^2+\frac{1}{2}|z_-|^2,$$
denote its moment map.

\begin{lemma}
If $a\!\in\!\R$ is a regular value for $h$ and $S^1$ acts freely on $V_a$,
$D_a$ is a symplectic submanifold of  $\tilde{h}^{-1}(a)/S^1$ with symplectic normal
bundle isomorphic~to $L_+\!\oplus\!L_-$.
\end{lemma}

\begin{proof}
By \cite[Theorem~1]{MW}, there are unique symplectic forms $\tilde\om$
and $\om'$ on $\tilde{h}^{-1}(a)/S^1$ and $D_a$, respectively, characterized by
$$\tilde\pi^*\tilde\om=(\om\oplus\om_0\oplus\om_0)\big|_{T(\tilde{h}^{-1}(a))}, \qquad \pi'^*\om'=(\om\oplus0\oplus0)|_{TV_a}
=(\om\!\oplus\om_0\oplus\om_0)\big|_{TV_a},$$
where $\tilde\pi\!:\tilde{V}_a=\tilde{h}^{-1}(a)\to\tilde{h}^{-1}(a)/S^1$ and
$\pi'\!:\tilde{h}^{-1}(a)\cap U\!\times\!0\!\times\!0\to D_a$ are the quotient maps.
Thus, $D_a$ is a symplectic submanifold of $\tilde{h}^{-1}(a)/S^1$.
With $\ep$ and $\Psi$ as in Lemma~\ref{tubular_lmm}, the map
\begin{gather*}
\widetilde\Psi\!:\big\{[x,z_+,z_-]\!\in\!L_+\!\otimes\!L_-\!:\,
\big||z_+|^2\!-\!|z_-|^2\big|\!<\!2\ep(x)\big\}\to 
\big(\tilde{V}_a\cap V_a(\ep)\!\times\!\C\!\times\!\C\big)/S^1,\\
\qquad [x,z_+,z_-]\to \big[\Psi(x,(|z_+|^2-|z_-|^2)/2),z_+,z_-\big],
\end{gather*}
is a diffeomorphism such that 
$$\{\widetilde\Psi^*\tilde\om\}(w_1,w_2)=
\{\om_0\oplus\om_0\}(w_1,w_2)\qquad\forall~w_1,w_2\in\C\oplus\C
\subset T_x(L_+\oplus L_-),~x\in D_a.$$
This implies the claim.
\end{proof}

Let $U_{\pm} = V_{a;\pm}/S^1$. Define $X_{\cut}$ to be the closed symplectic manifold obtained by gluing the open charts $U_{\pm}$ and $X\setminus V_a$ via the symplectic gluing map
\begin{equation}
 U_{\pm(h-a)>0} \subset X\setminus V_a \to U_{\pm}\setminus D,\quad x \to (x,\sqrt{\pm(h(x)-a)}) \in V_{a;\pm} \stackrel{\tn{proj}}{\longrightarrow} U_{\pm}, 
\end{equation}
on the overlap. The space $X_{\cut}$ is called the \textsf{symplectic cut} of $X$ along $V_a$ with respect to the Hamiltonian $S^1$-action. If $V_a$ separates $X$ into two connected components, then $X_{\cut}$ is a union of two symplectic manifolds $(X_{\pm},\om_{\pm})$ with $X_{\pm}$ containing $U_{\pm}$.

Following \cite{A}, we apply the symplectic cut construction to $T^*S^n$ with the $S^1$-action on $U=T^*S^n\setminus S^n$ corresponding to the Hamiltonian 
$$ h\colon U\to \R,\qquad h(\alpha)=|\alpha|,$$
where $|\alpha|$ is the norm of $\alpha\in T^*S^n$ with respect to the standard metric on $S^n \subset \R^{n+1}$. Let $\ze_h$ be the corresponding vector field on $U$. If $x_i\colon S^n \to \R$ are the normal coordinates around a point $p\in S^n$ and $y_i\colon T^*S^n\to \R$ are the induced functions on $T^*S^n$, then
$$\ze_h(x,y)=\frac{1}{|y|} \sum_{i=1}^{n}y_i \frac{\partial}{\partial x_i}+ O(|x|).$$
Thus, if $\gamma\colon(-\ep,\ep)\to (U,\alpha)$ is a $\ze_h$-trajectory, with $\alpha \in T^*_vS^n-0$, then 
$$ \gamma(t)^*=\frac{\tn{d}}{\tn{d}t} \tn{exp}_v(t\alpha^*/|\alpha|)\qquad \forall t\in (-\ep,\ep),$$
where $\gamma(t)^*,\alpha^*\in T^*S^n$ are the duals of $\gamma(t),\alpha\in T^*S^n$ with respect to the metric on $S^n$ and $\exp$ is the 
exponential map for this metric. Thus, the orbits of $\ze_h$ are $2\pi$-periodic and describe an $S^1$-action on $T^*S^n$. With respect to the identification
\fontsize{10}{12}\selectfont
\begin{equation}\label{equ1}
 U\cong \left\{ (v,w,r)\in S^n\times S^n \times \R^+: v\cdot w=0\right\} \subset \R^{n+1}\times \R^{n+1} \times \R \cong \C^{n+1}\times \R,
\end{equation}
 \normalsize
this action is the restriction of the action 
$$e^{\mf{i}\theta}(x,r)=(e^{-\mf{i}\theta}x,r).$$
In particular, this action is free. For any $a\in \R^+$,
\fontsize{10}{12}\selectfont
\begin{equation*}
\begin{split}
V_{-;a}&\cong \left\{(x_1,\cdots,x_{n+1},r,z)\in \C^{n+1}\times \R^+\times \C: \sum_{i=1}^{n+1}x_i^2=0,~\sum_{i=1}^{n+1}|\tn{Re}~x_i|^2=1,~r+\frac{|z|^2}{2}=a \right\}\\
& \cong  \left\{(x_1,\cdots,x_{n+1},z)\in \C^{n+1}\times \C: \sum_{i=1}^{n+1}x_i^2=0,~\sum_{i=1}^{n+1}|\tn{Re}~x_i|^2=1,~|z|^2<2a \right\}\equiv V'_{-,a}
\end{split}
\end{equation*}
\normalsize
with respect to the identification (\ref{equ1}) and 
$$ S^1\times V'_{-,a} \to V'_{-,a},\qquad e^{\mf{i}\theta}\cdot(x_1,\cdots,x_{n+1},z)=(e^{-\mf{i}\theta}x_1,\cdots,e^{-\mf{i}\theta}x_{n+1},e^{-\mf{i}\theta}z).$$
The symplectic manifold $((T^*S^n)_-,\om_-)$ is obtained by gluing 
\begin{equation}\label{equ2}
\begin{split}
(T^*S^n)_{<a}&\equiv \left\{\alpha \in T^*S^n \colon |\alpha|<a\right\}\\
             &\cong\left\{(v,w)\in S^n\times \R^{n+1}\colon v\cdot w=0, |w|<a\right\} \subset \C^{n+1}
\end{split}
\end{equation}
\normalsize
with $V'_{-;a}/S^1$ by the map 
\begin{equation}\label{equ3}
 (T^*S^n)_{<a}~-~S^n \to V'_{-;a}/S^1,\quad (v,w)\to [v+\mf{i}\frac{w}{|w|},\sqrt{2(a-|w|)}].
\end{equation}
We define a diffeomorphism 
$$f\colon (T^*S^n)_-\to Q^n\equiv \left\{[x_0,\cdots,x_{n+1}]\in \P^{n+1}\colon x_0^2=\sum_{i=1}^{n+1}x_i^2\subset \P^n\right\}$$
by 
\begin{equation*}
\begin{split}
(T^*S^n)_{<a}&\to Q^n,\quad (v,w)\to\left[\sqrt{1-\frac{|w|^2}{a^2}},v+\mf{i}\frac{w}{a}\right],\\
 V'_{-;a}/S^1 &\to Q^n,\quad [x,z]\to\left[\frac{z}{\sqrt{a}}\sqrt{1-\frac{z^2}{4a}}, \left(1-\frac{|z|^2}{4a}\right)x+\frac{|z|^2}{4a}\bar{x}\right].
\end{split}
\end{equation*}
By (\ref{equ3}), $f$ is well-defined on the overlap. Using the fact that 
$$\om_{T^*S^n}=\om_{\C^{n+1}}\!|_{T^*S^n}$$
with respect to the identification (\ref{equ2}), we find that
$$f^*\om_{\tn{FS}}=\frac{\pi}{4a}\om_-.$$
The diffeomorphism $f$ takes $S^n\subset (T^*S^n)_-$ and 
$$D=\{[x_1,\cdots,x_{n+1},z]\in V'_{-;a}/S^1 :z=0\}$$
to $Q^n_{\R}\equiv Q^n\cap \R\P^n$ and $Q^n\cap (x_0=0)\cong Q^{n-1}$, respectively. The former is the fixed-point locus of the restriction $\tau_{n}^Q$ to $Q^n$ of the standard involution 
$$\tau_{n+1}\colon \P^{n+1}\to\P^{n+1},\qquad [z_0,\cdots,z_{n+1}]\to[\bar{z}_0,\cdots,\bar{z}_{n+1}].$$
This is the involution on $Q^n$ induced by the standard involution of $T^*S^n$ via the diffeomorphism $f$.

A free $\Z_k$ action on $S^n$ by isometries, with primitive $k$-th roots $\xi_1,\cdots,\xi_m$ of 1 if $n=2m-1$, commutes with the above $S^1$-action and thus induces a Hamiltonian action on $T^*L$, where $L=S^n/\Z_k$. It also induces an action on $Q^n$ taking $[z_0,\ldots,z_{n+1}]$ to
$$
\begin{cases} [z_0,-z_1,\ldots,-z_{n+1}],\!\!\!&\hbox{if}\,k\!=\!2,\\
[z_0,\xi_1^{\R}z_1\!+\!\xi_1^{\mf{i}\R}z_2,\xi_1^{\R}z_2\!-\!\xi_1^{\mf{i}\R}z_1,\ldots,
\xi_m^{\R}z_n\!+\!\xi_m^{\mf{i}\R}z_{n+1},\xi_m^{\R}z_{n+1}\!-\!\xi_m^{\mf{i}\R}z_n],\!\!\!&
\hbox{if}\,n\!=\!2m\!-\!1,
\end{cases}
$$
where $\xi_i^{\R}$ and $\xi_{i}^{\mf{i}\R}$ are the real and imaginary parts of $\xi_i$, respectively. This action preserves the divisor 
$$D_n\equiv \{[z_0,\cdots,z_{n+1}]\in Q^{n}\colon z_0=0\}\equiv Q^{n-1}$$
and commutes with the involution $\tau_{n}^{Q}$. The quotient $Q^n/\Z_k$ and $D_n/\Z_k$ thus correspond to a symplectic cut on $T^*L$ for the $S^1$-action on $L$ and inherit involutions from the standard involution on $T^*L$. If $k=2$, the $S^1$-action on $S^n$ induces a free $S^1/\Z_2\cong S^1$ Hamiltonian action on $L=\R\P^n$ with
\begin{equation*}
\begin{split}
((T^*L)_-,L,D)\cong (Q^n,Q^n_{\R},D_n)/\Z_2\cong (\P^n,\R\P^n,D_n),\\
[x_0,x_1,\cdots,x_{n+1}]\to [x_1,\cdots,x_{n+1}].
\end{split}
\end{equation*}
If the induced $\Z_k$-action on $Q^n$ is free, then the $S^1$-action on $L=T^*S^n/\Z_k$ is also free and 
$$ ((T^*L)_-,L,D)\cong (Q^n,Q^n_{\R},D_n)/\Z_k.$$

\subsection{Symplectic sum}\label{sec:sum}
In this section, we review the symplectic sum construction, following \cite[Section 2]{IP2}. Starting from a nodal symplectic manifold $(X_{\cut},\om_{\cut})$ as in Section \ref{sec:cut}, we build a symplectic fibration $\pi\colon\mc{X}\to \De$ with central fiber $\mc{X}_0$ obtained from $X_{\cut}$ by identifying the two copies of $D$ in the canonical way. Since the operation is localized near $D\subset X_{\cut}$, we may assume that $X_{\cut}=X_-\sqcup_D X_+$; therefore, throughout this section $(X_{\pm},D)$ are two symplectic manifolds each containing a copy of a symplectic divisor $D$ so that their normal bundles $\mc{N}^{X_{\pm}}_{D}$ are dual to each other. The symplectic sum construction involves gluing three open symplectic charts: 
$$\mc{X}_{\pm}=(X_{\pm}\setminus D)\times \De,~~~ 
\mc{X}_{\tn{neck}}= \{ (p,x,y)\!\in\!\mc{N}^{X_{+}}_{D}\!\oplus\!\mc{N}^{X_{-}}_{D}\colon~~ |x|,|y|\leq1, |xy|<\de \},$$
where $\De\subset\C$ is a small disk, $\de\in\R^+$ is sufficiently small, and $|\cdot|$ denotes a Hermitian norm on $\mc{N}_D^{X_+}$
and the dual Hermitian norm on $\mc{N}_D^{X_-}\cong(\mc{N}^{X_+}_D)^*$.

Given a complex line bundle $\pi\colon E \to D$, fix a hermitian metric on $E$. Let
$$  \rho\colon E\to \R,\quad \rho(x)=\frac{1}{2}|x|^2,\qquad  \rho^*\colon E^*\to \R,\quad \rho^*(y)=\frac{1}{2}|y|^2.$$ 
A Hermitian connection in $E$ defines a 1-form $\alpha$ on $E\setminus D$ with $\alpha(\partial{\theta})=1$. This is the pull-back of the connection form on the circle bundle viewed as a principal $S^1$-bundle. On the total space of $E$, we define a symplectic form~by
$$ \om_{E}= \pi^*\om_D + d(\rho \alpha).$$
It extends across the zero section and is $S^1$-invariant. The dual bundle $E^*$ inherits a dual connection $\alpha^*$.
Hence, we get a symplectic form on $\pi \colon E\oplus E^* \to D$,  
\begin{equation*}
\begin{aligned}
\om_{\tn{neck}}&= \pi^*\om_D+ d(\rho \wedge \alpha)+d(\rho^* \wedge \alpha^*)\\
						   &=  \pi^*\om_D+ (\rho-\rho^*) \pi^* F  +d\rho \wedge \alpha + d\rho^* \wedge \alpha^*,
\end{aligned}
\end{equation*}						   
where $F$ is the curvature 2-form of $\alpha$. This space admits an $S^1$-action given by 
$$(p,x,y) \to (p,e^{\mf{i}\theta} x, e^{-\mf{i}\theta} y)$$
with moment map $\rho^*-\rho$. There is also a natural $S^1$-invariant map 
$$\la\colon E\oplus E^* \to \C,\qquad (p,x,y)\to xy \in \C.$$ 
Putting $E=\mc{N}^{X_{+}}_{D}$, we obtain a symplectic form $\om_{\tn{neck}}$ on $\mc{X}_{\tn{neck}}$. 

Let $\mc{X}$ be the smooth manifold obtained by gluing the three charts by the diffeomorphisms
\begin{equation}\label{chart-gluing}
\psi_{\pm} \colon \mc{X}_{\tn{neck}}\setminus \mc{N}_D^{X_{\mp}} \to \mc{X}_{\pm},  \quad (p,x_+,x_-) \to ((p,x_{\pm}),x_+x_-).
\end{equation}
The map $\la\colon \mc{X}_{\tn{neck}} \to \C$ extends to the whole of $\mc{X}$ and gives $\mc{X}$ the structure of a fibration over $\De$ such that the fiber over zero is $X_{-}\cup_D X_{+}$ and the other fibers are smooth manifolds.

\begin{remark}
If the manifolds $X_{\pm}$ are obtained from the symplectic cut procedure on $(X,\om)$ along $V_a=h^{-1}(a)\subset U \subset X$, then $\mc{X}_{\pm}\cong X_{\pm(h-a)>0} \times \De$ and $\mc{X}_{\tn{neck}}\cong \tilde{h}^{-1}(a)/S^1$, where
$$\tilde{h}\colon U\times \C \times \C \to \R, \quad \tilde{h}(p,x,y)=h(p)-\frac{1}{2}|x|^2+\frac{1}{2}|y|^2,$$
is the moment map for the $S^1$-action $(p,x,y)\to (p,e^{\mf{i}\theta} x, e^{-\mf{i}\theta} y)$ and $\om_{\tn{neck}}$ is the symplectic structure induced from $\om\oplus \om_0\oplus \om_0$ via symplectic cut. 
\end{remark}

We next define a symplectic structure on $\mc{X}$. By the Symplectic Neighborhood Theorem \cite[Chapter 3]{MS1}, a neighborhood of $D$ in $X_{\pm}$ is symplectomorphic to the disc bundle of radius $\ep\leq 1$ in $\mc{N}_D^{X_{\pm}}$. We can assume $\ep=1$ (by some re-scaling). Let $\om_0= r dr d\theta$ be the symplectic form on $\C$. In the overlap region, where $1-\de < |x| < 1$ and $|y|<\de$, 
$$\om_{\tn{neck}}= \om_{X_{+}}+ d(\rho^* \alpha^*)\quad \tn{and} \quad \psi_+^* (\om_{X_+} \oplus \om_0) = \om_{X_{+}} + \la^* \om_0,$$ 
because of the symplectic neighborhood identification. We can smoothly merge  
$$\la^* (\om_0)= \frac{1}{2} d (|\la|^2 \la^* d\theta) = 2d(\rho \rho^* \la^* d\theta)=d(2\rho\rho^* (\alpha+\alpha^*))$$ 
into $d(\rho^* \alpha^*)$ by replacing $ 2\rho^* \rho (\alpha+\alpha^*)$ with 
$$\eta\cdot 2\rho^* \rho (\alpha+\alpha^*)+(1-\eta)\cdot \rho^* \alpha^*,$$ 
where $\eta=\eta(|x|)$ is a cutoff function such that $\eta(|x|)=1$ if $|x|\geq 1$, $\eta(x)=0$ if $|x|\leq 1-\de$, and $|d\eta| < 2/\de$. If $\de$ is sufficiently small, the closed two-form
$$\om_{X_{+}}+ d(\eta(2\rho^* \rho (\alpha+\alpha^*))+(1-\eta) \rho^* \alpha^*)$$ 
is non-degenerate; see \cite[Section 2]{IP2}. We can do the same procedure for the other overlap, thus obtaining a symplectic form $\om_{\mc{X}}$ on $\mc{X}$. 

Suppose $c_1(TX)=0$. Let $A\in H_2(X_+)$ and $s=A\cdot D$. After multiplying $A$ by some scalar, there is $B\in H_2(X_-)$ such that $B\cdot D=s$. We can deform $A+B$ into a homology class $C=A\#B$ in the smooth fibers $\mc{X}_{\la}$. By \cite[Lemma 2.4]{IP2}, 
$$0= K_{\mc{X}_{\la}}(C)= K_{X_-}(B)+K_{X_{+}}(A)+2s.$$
In each case, $K_{X_-}$ is equal to $\frac{1-(k_0+n)}{k}\cdot[D]$, where $k_0=1+\de_{k,2}$ is
the order of the branching of $Q^n\ra X_-$ along $Q^{n-1}$, since
$$K_{Q^{n}}=\pi^* K_{X_{-}}+(k_0-1)[Q^{n-1}]$$
and $K_{Q^{n}}=-n [Q^{n-1}]$. This confirms (3) in Proposition~\ref{prop:main-space}.

The statement of Proposition~\ref{prop:main-space} concerning $\mc{J}_{\mc{X}}$ follows from \cite[Lemma 2.3]{IP2} and \cite[Theorem A.2]{IP1}.

\subsection{Involution on the symplectic cut and sum}\label{sec:involution}
It remains to prove the claim of Proposition \ref{prop:main-space} concerning antisymplectic involutions. We call an antisymplectic involution $\phi$ and an $S^1$-action $e^{\mf{i}\theta}\colon U \to U$ with moment map $h\colon U\to \R$ \textsf{compatible} if $h=h\circ \phi$.
Since the Hamiltonian flow for $h\circ\phi$ is $\phi\circ e^{-\mf{i}\theta}\circ \phi$,
 $\phi \circ e^{\mf{i}\theta} = e^{-\mf{i}\theta}\circ\phi$ for a $\phi$-compatible $S^1$-action. 
 
If $D\subset X$ is a submanifold preserved by an involution $\phi$ on $X$, the differential $d\phi$ induces a linear map 
$$\phi_*\colon \mc{N}_D^X\to \mc{N}_D^X, \qquad v\to d\phi(v)+T_{\phi(p)}D\quad \forall v\in T_pX, p\in D,$$
covering $\phi\colon D\to D$.

\begin{lemma}\label{lem:ind-invol}
Let $(X,\om,\phi)$ be a symplectic manifold with a real structure, $h\colon U\to S^1$ be the moment map for a free Hamiltonian $S^1$-action on an open subset $U\subset X$, and $a\in\R$ be a regular value of $h$ so that the hypersurface $V_a\equiv h^{-1}(a)$ is non-empty. Let $X_{\cut}$ be the corresponding symplectic manifold obtained from $(X,\om)$ by symplectically cutting along $V_a$ as in Section~\ref{sec:cut} and let $D_{\pm}\cong D \subset X_{\cut}$ be the corresponding divisors. If the $S^1$-action is compatible with $\phi$, there is a real structure $\phi_{\cut}$ on $X_{\cut}$ preserving $D_{\pm}$ such that the canonical projection map $X\to X_{\cut}/D_+\sim D_-$ intertwines the involutions $\phi$ and $\phi_{\cut}$ and
$$
(\mc{N}_{D_+}^{X_{\cut}}\otimes \mc{N}_{D_-}^{X_{\cut}},(\phi_{\cut})_*\otimes (\phi_{\cut})_*) \cong (D\times \C, \phi_{\cut} \times c),
$$
where $c$ is the standard complex conjugation on $\C$.
If in addition $\Fix(\phi)\cap V_a=\emptyset$, then $\Fix(\phi_{\cut})\cap D_{\pm}=\emptyset$.
\end{lemma}

\begin{proof}
We continue with the notation of the symplectic cut construction in Section~\ref{sec:cut}. We define $\phi_{\cut}$ on $X_{\cut}$ by
$$ \phi_{\cut}\colon X\setminus V \to  X\setminus V,\quad x\to \phi(x), \qquad \phi_{\pm}\colon U_{\pm}\to U_{\pm},\quad [x,z]\to[\phi(x),\bar{z}].$$
Since the moment maps $h_{\pm}$ are invariant with respect to the involution $(x,z)\to(\phi(x),\bar{z})$ on $U\times \C$ and $\phi \circ e^{\mf{i}\theta}=e^{-\mf{i}\theta}\circ \phi$, the second map above is well-defined. Since it is induced by an antisymplectic map on $U\times \C$, it is antisymplectic. 

If $\phi$ preserves an orbit $S^1\cdot p \subset V_a$, then $\phi(p)=e^{\mf{i}a}\cdot p$ for some $e^{\mf{i}a}\in S^1$ and 
$$\phi(e^{\mf{i}\theta}\cdot p)=e^{\mf{i}(a-2\theta)}\cdot e^{\mf{i}\theta}\cdot p \qquad \forall e^{\mf{i}\theta}\in S^1.$$
Thus, $e^{\mf{i}a/2}\cdot p \in \Fix(\phi)\cap V_a$. This implies the last claim.
\end{proof}

\begin{remark}
If $V_a$ separates $X$ into two connected components, then $\phi_{\cut}$ restricted to $X_{\pm}$ is an involution $\phi_{\pm}$ agreeing on the common divisor $D$. 
\end{remark}

\begin{lemma}\label{deforming-involution0}
Let $(X,\om,\phi)$ be a symplectic manifold with a real structure and $L=\Fix(\phi)$.
Then there exist a neighborhood $N(L)\subset T^*L$ of the zero section, a neighborhood $U\subset X$ of $L$, and a diffeomorphism
$$  \psi\colon(N(L),L) \to (U,L) \qquad \tn{s.t.} \qquad \psi^*\om =\om_L,\quad \psi^{-1}\circ \phi\circ \psi= \tau_L,$$
where $\om_L$ and $\tau_L$ are the canonical symplectic form and antisymplectic involution on $T^*L$, respectively.
\end{lemma}

\begin{proof}
The proof is a modification of the proofs of \cite[Theorem~3.33]{MS1} and \cite[Lemma~3.14]{MS1}.
Let $J$ be an $(\om,\phi)$-compatible almost complex structure on $X$ and denote by $g_J$ the associated metric. Let
$$ \Phi_q \colon T_q^*L \to T_q L,\qquad g_J(\Phi_q(v^*),v)=v^*(v)\quad \forall v^* \in T_q^*L,\; v\in T_qL, $$
be the isomorphism induced by the metric $g_J$. Define 
$$ \psi\colon T^*L \to X \quad \hbox{by}\quad \psi(q,v^*) = \exp_q(J_q \Phi_q(v^*)).$$
Since $\phi$ is an isometry of $g_J$, 
$$ \phi(\exp_q(J_qu))=\exp_q(D\phi(J_qu))= \exp_q(-J_qD\phi(u))=\exp_q(-J_qu)\quad \forall v\in T_qL.$$
Therefore, $\psi^{-1} \circ \phi \circ \psi(v^*)=(-v^*)= \tau_{L}(v^*)$. By the proof of \cite[Theorem~3.33]{MS1}, $\psi^*\om|_{L}=\om_L$.

Define $\om_1=\psi^*\om$ and $\om_0=\om_L$; then $\om_0$ and $\om_1$ are two symplectic forms on $N(L)$ such that $\tau_L^*\om_i=-\om_i$.
By \cite[Lemma 3.14]{MS1}, there is a path of symplectomorphisms $\varphi_t$ such that $\varphi_1^*\om_1=\om_0$.
We show that $\varphi_t$ can be chosen to commute with $\tau_L$, i.e.
\begin{equation}\label{equ:compatibility}
\tau_L \circ \varphi_t =\varphi_t \circ \tau_L  .
\end{equation}
If $\si$ is as in \cite[(3.7)]{MS1}, $d\si = \om_1 -\om_0$ and $\alpha \equiv \frac{\si+\tau_L^*\si}{2}$ is a $\tau_L$-invariant closed $1$-form. Replacing $\si$ by $\si-\frac{\alpha}{2}$,
we can assume that $\tau_L^*\si=-\si$. This implies that the path of symplectomorphisms $\varphi_t$ given by the vector field $X_t$ such that
$$ \si = \iota_{X_t}(t\om_1 +(1-t) \om_0)$$
satisfies (\ref{equ:compatibility}).
\end{proof}

\begin{corollary}\label{coro:main-coro}
Let $(X,\om,\phi)$ be a symplectic manifold with a real structure and $L=\Fix(\phi)$. If $L$ is an archetypal Lens space, there is a symplectic cutting of $X$ into symplectic manifolds $(X_{\pm},\om_{\pm})$ with antisymplectic involutions $\phi_{\pm}$  and $\phi_{\pm}$-invariant symplectic divisor $D$ so that $\Fix(\phi_+)=\emptyset$, $\Fix(\phi_-)=L$, and there are isomorphisms
\begin{equation*}
\begin{split}
(X_-,\om_-,D,L)&\cong (Q^n/\Z_k,\om_{\FS},D_n/\Z_k,Q^n_{\R}/\Z_k),\\  
(\mc{N}_D^{X_+}\otimes \mc{N}_{D}^{X_-},(\phi_+)_*\otimes (\phi_-)_*) &\cong (D\times \C, \phi_{\pm} \times c).
\end{split}
\end{equation*}
\end{corollary}

\begin{proof}
If $L$ is a Lens space and $U$ is a Weinstein neighborhood as in Lemma \ref{deforming-involution0}, 
then $\phi\!|_U\cong \tau_L$ is compatible with the associated $S^1$-action of $L$ described in the second half of Section~\ref{sec:cut}
and therefore descends to the symplectic cut by Lemma~\ref{lem:ind-invol}.
\end{proof}

We next obtain a statement similar to Lemma~2.11 for $\phi$-invariant symplectic submanifolds. 

\begin{lemma}\label{lem:sym-neighborhood}
Let $(X,\om,\phi)$ be a symplectic manifold with a real structure and $D\!\subset\!X$ a symplectic divisor preserved by~$\phi$.
Then there exist a neighborhood $N(D)\subset \mc{N}_D^X$ of the zero section, a neighborhood $U\subset X$ of $D$, and a diffeomorphism
$$\psi\!:(N(D),D)\ra(U,D) \qquad\textnormal{s.t.}\qquad \psi^{-1}\circ\phi\circ\psi=\phi_*: N(D)\ra N(D).$$
\end{lemma}

\begin{proof}
The proof is a modification of the proof of \cite[Theorem 3.30]{MS1}. Let $J$ be an $(\om,\phi)$-compatible almost complex structure on $X$ and denote by $g_J$ the associated metric. There is an isomorphism
$\mc{N}_D^X\cong TD^{\om}$,
where $TD^{\om}$ is the $\om$-orthogonal complement of $TD$ in $TX\!|_{D}$. Let $\exp\colon TD^{\om} \to X$ be the exponential map associated to $g_J$.
Since $\phi$ is an isometry with respect to~$g_J$, 
$$\phi\big(\exp(v)\big) = \exp\big(\phi_*(v)\big) \qquad\forall\,v\in TD^{\om}.$$ 
The restriction $\psi$ of $\tn{exp}$ to some neighborhood $N(D)$ of $D\subset TD^{\om}$ is a diffeomorphism onto an open subset $U\subset X$. 
\end{proof}

\begin{lemma}
Let $(X_{\pm},\om_{\pm},\phi_{\pm})$ be symplectic manifolds with real structures and $D\subset X_{\pm}$ be a common symplectic divisor preserved by $\phi_{\pm}$ such that $\phi_-\!|_D=\phi_+\!|_D$. If 
\begin{equation}\label{equ:isomorphism}
(\mc{N}_D^{X_+}\otimes \mc{N}_{D}^{X_-},(\phi_+)_*\otimes (\phi_-)_*) \cong (D\times \C, \phi_{\pm} \times c),
\end{equation}
where $c$ is the standard complex conjugation on $\C$, the corresponding symplectic sum fibration $\mc{X}\to \De$ as in Section~\ref{sec:sum} can be constructed so that it admits an antisymplectic involution $\phi_{\mc{X}}$ such that $\phi_{\mc{X}}\!|_{X_{\pm}}=\phi_{\pm}$.
\end{lemma}

\begin{proof}
For the purposes of the symplectic sum construction, we identify neighborhoods of $D$ in $X_{\pm}$ and in $\mc{N}_{D}^{X_{\pm}}$ as in Lemma~\ref{lem:sym-neighborhood} and use an isomorphism as in (\ref{equ:isomorphism}).
With notation as in Section~\ref{sec:sum}, the involutions $\phi_{\pm}$ on $X_{\pm}$ extend to $\mc{X}$ by 
\begin{alignat*}{2}
 \mc{X}_{\pm}      &\to \mc{X}_{\pm},          &\qquad (p,z)       &\to (\phi_{\pm}(p),\bar{z}),\\
 \mc{X}_{\tn{neck}}&\to\mc{X}_{\tn{neck}},    &\qquad (p,x_+,x_-)     &\to (\phi_{\pm}(p),(\phi_+)_*x_+,(\phi_-)_*x_-).
\end{alignat*}
By (\ref{chart-gluing}) and (\ref{equ:isomorphism}), these involutions agree on the overlaps and are intertwined by $\la$ with the conjugation on $\De$. 
For the resulting involution $\phi_{\mc{X}}$ on $\mc{X}$ to be compatible with the symplectic structure on $\mc{X}$, we choose the bump function $\eta$ used in the merging procedure to be symmetric with respect to $\phi_{\mc{X}}$. 
\end{proof}

The statement of Proposition~\ref{prop:main-space} concerning $\mc{J}_{\phi_{\mc{X}}}$ follows from the proofs of \cite[Lemma 2.3]{IP2} and \cite[Theorem A.2]{IP1}, since each step in the proofs is compatible with the involution.

A $(D,\phi)$-compatible almost complex structure on $\mc{X}$ can also be constructed by viewing $X_{\pm}$ as symplectic cuts of $X=\mc{X}_{\la}$
for some $\la\in\De^*$. Start with an almost complex structure $J$ on $X$ which is compatible with the involution and the $S^1$-action on $U$. We know that $\mc{X}_{\tn{neck}}= \tilde{h}^{-1}(a)/S^1$, where $$\tilde{h}\colon U\times \C \times \C \to \R,\qquad \tilde{h}(p,x,y)=h(p)-\frac{1}{2}|x|^2+\frac{1}{2}|y|^2,$$
is the moment map. 
The almost complex structure $J\oplus j\oplus j$ on $U\times\C\times\C$, where $j$ is the standard complex structure on $\C$, 
induces an almost complex structure $J_{\tn{neck}}$ on $\mc{X}_{\tn{neck}}$, which has the required compatibility properties. 
There is also a natural extension of $J$ to an almost complex structure $J_{\pm}$ on~$\mc{X}_{\pm}$. 
Merging the corresponding metrics, $\om_{\tn{neck}}(\cdot,J_{\tn{neck}}\cdot)$ and $\om_{\pm}(\cdot,J_{\pm}\cdot)$,
away from $D$ and applying the polarization procedure of \cite[Appendix]{IP1}, 
we get an almost complex structure on the total space.\\

\section{J-holomorphic discs and open GW invariants\label{ch:holomorphicdiscs}}

Throughout this section $(X,\omega,\phi,L)$ denotes a symplectic manifold equipped with an antisymplectic involution $\phi$ whose fixed-point set is a Lagrangian $L$. We assume that $L$ is orientable and spin and fix an orientation and a spin structure $\sigma$ on $L$. In this case, open invariants are defined in \cite{S} using perturbed Cauchy-Riemann equations. In Section~\ref{open-invariants-review} below, we review the construction of these invariants in the language of Kuranishi structures. In Section~\ref{relative-open-invariants}, we outline the construction of a relative version of open invariants.

\subsection{Review of open GW invariants}\label{open-invariants-review}

Let $\mc{J}_{\phi}$ be the space of $(\om,\phi)$-compatible almost complex structures. For $J\in \mc{J}_{\phi}$, the involution $\phi$ on $X$ induces an involution
\begin{equation} \label{tau-M}
\begin{split}
&\tau_{\mc{M}} \colon \mc{M}_{k,l}^{\disc}(X,L,\beta) \to \mc{M}_{k,l}^{\disc}(X,L,\beta),\\
& \tau_{\mc{M}} ([u,\vec{z},\vec{w}]) = [\phi\circ u \circ c,(\ov{z}_1,\ov{z}_{k},\ov{z}_{k-1},\cdots,\ov{z}_2), (\bar{w}_1,\cdots,\bar{w}_l)],
\end{split}
\end{equation}
where $c(z)=\bar{z}$. It naturally extends to maps with bubble domain, inducing an involution on $\ov{\mc{M}}_{k,l}^{\disc}(X,L,\beta)$. We call a Kuranishi structure on $\ov{\mc{M}}_{k,l}^{\disc}(X,L,\beta)$ \textsf{$\tau_{\mc{M}}$-invariant} if $\tau_{\mc{M}}$ extends to a map on Kuranishi neighborhoods and multisections. 


\begin{proposition}[{\cite{FO-3},\cite[Chapter 7]{FOOO}\label{kur-str-1}}]
Let $(X,\omega,\phi)$ be a symplectic manifold with a real structure. The moduli space $\ov{\mc{M}}_{k,l}^{\disc}(X,L,\beta)$ has a topology with respect to which it is compact and Hausdorff. It has a $\tau_{\mc{M}}$-invariant oriented Kuranishi structure with boundary and with virtual real dimension $$\dim^{\vir}(\ov{\mc{M}}_{k,l}^{\disc}(X,L,\beta))=\dim_{\C}X+\mu(\beta)+k+2l-3.$$ 
The codimension one boundary components of $\ov{\mc{M}}^{\disc}_{k,l}(X,L,{\beta})$ are described by (\ref{boundary-equ1}) and (\ref{boundary-equ2}) in a way that respects the Kuranishi structures. A spin structure $\sigma$ on $L$  determines an orientation of $\ov{\mc{M}}_{k,l}^{\disc}(X,L,\beta)$.
\end{proposition}

Let $\ov{\mc{M}}_{k,l}^{\disc}(X,L,\beta)^{\sigma}$ denote the moduli space equipped with the orientation induced by $\sigma$. We are interested primarily in manifolds of real dimension six. Since the tangent bundle of every orientable manifold $L$ of dimension three is trivial, $L$ is automatically spin. A choice of trivialization of the tangent bundle of $L$ determines an orientation on $\mc{M}_{k,l}^{\disc}(X,L,\beta)$. Therefore, in this case by a spin structure we simply mean a choice of trivialization of $TL$.

If $(X,L)$ has vanishing Maslov class and real dimension six, then 
$$\ov{\mc{M}}^{\disc}(X,L,\beta)\equiv \ov{\mc{M}}^{\disc}_{0,0}(X,L,\beta)$$ 
has virtual dimension zero. We would like to define invariants by counting the number of elements in $\ov{\mc{M}}^{\disc}(X,L,\beta)$. Fix a choice of $\tau_{\mc{M}}$-invariant multisection~$\mf{s}$ (see \cite[Section 7]{S}) whose zero locus is close to that of the Kuranishi map. Let $[\ov{\mc{M}}^{\disc}(X,L,\beta)^{\sigma}]^{\mf{s}}$ be the virtual fundamental class determined by the multisection~$\mf{s}$. Since the moduli space is zero-dimensional, its degree is a rational number, which we denote by $N^{\disc}_{\beta,J,\mf{s}}$; a priori it depends on $J$ and~$\mf{s}$.

Given two different choices of $(J_i,\mf{s}_{i})$, $i=0,1$, let $\left\{J_t \in \mc{J}_{\phi}\right\}$, $t \in [0,1]$, be a path of almost complex structures joining $J_0$ and $J_1$. Let 
$$\pi \colon \ov{\mc{M}}^{\disc}(X,L,\left\{J_t\right\},\beta)\equiv \coprod_{t\in[0,1]} \ov{\mc{M}}^{\disc}(X,L,J_t,\beta)  \to [0,1]$$
be the projection map. There is an analogue of Proposition \ref{kur-str-1} for $\ov{\mc{M}}^{\disc}(X,L,\left\{J_t\right\},\beta)$. In fact, $\partial^1 \ov{\mc{M}}^{\disc}(X,L,\left\{J_t\right\},\beta)$ is a union of $\ov{\mc{M}}^{\disc}(X,L,J_i,\beta)$ and the boundary terms of the form (\ref{boundary-equ1}) and (\ref{boundary-equ2}). 

Choose a $\tau_{\mc{M}}$-invariant multisection $\mf{s}$ for $\ov{\mc{M}}^{\disc}(X,L,\left\{J_t\right\},\beta)$ such that $\mf{s}|_{\pi^{-1}(i)}=\mf{s}_i$ and let $[\ov{\mc{M}}^{\disc}(X,L,\left\{J_t\right\},\beta)^{\sigma}]^{\mf{s}}$ be the one-dimensional fundamental chain of $\mf{s}$. Then,
$$\partial{[\ov{\mc{M}}^{\disc}(X,L,\left\{J_t\right\},\beta)^{\sigma}]^{\mf{s}}} = [\partial{\ov{\mc{M}}}^{\disc}(X,L,\left\{J_t\right\},\beta)^{\sigma}]^{\mf{s}}$$
and 
\begin{equation}\label{boundary}
\begin{split}
&N^{\disc}_{\beta,J_1,\mf{s}_1}-N^{\disc}_{\beta,J_0,\mf{s}_0}= \sum_{\tilde{\beta}\in j^{-1}(\beta)}\# [\mc{M}_1(X,\left\{J_t\right\},\tilde\beta)\times_{\ev_1} L]^{\mf{s}} \\
&+ \sum_{\beta_1+\beta_2=\beta} \# [\mc{M}_{1,0}^{\disc}(X,L,\left\{J_t\right\},\beta_1)^{\sigma} \times_{(\ev^B_1,\ev^B_1)} \mc{M}_{1,0}^{\disc}(X,L,\left\{J_t\right\},\beta_2)^{\sigma}]^{\mf{s}}.
\end{split}
\end{equation}
We would like to see if the right-hand side of this equation vanishes. For this, define an involution $\tau_{\mf{glue}}$ on $\mc{M}_{1,0}^{\disc}(X,L,\left\{J_t\right\},\beta_1) \times_{(\ev^B_1,\ev^B_1)} \mc{M}_{1,0}^{\disc}(X,L,\left\{J_t\right\},\beta_2)$ by
$$ (u_1,u_2)\to (u_1,\tau_{\mc{M}}(u_2)).$$
\begin{remark}
Every 
$$(u_1, u_2) \in \mc{M}_{1,0}^{\disc}(X,L,\left\{J_t\right\},\beta_1)\times_{(\ev^B_1,\ev^B_1)} \mc{M}_{1,0}^{\disc}(X,L,\left\{J_t\right\},\beta_2)$$  is also an element 
$$(u_2, u_1)\in \mc{M}_{1,0}^{\disc}(X,L,\left\{J_t\right\},\beta_2)\times_{(\ev^B_1,\ev^B_1)} \mc{M}_{1,0}^{\disc}(X,L,\left\{J_t\right\},\beta_1).$$ Therefore, $\tau_{\mf{glue}}$ is not well-defined by the above. In order to avoid this ambiguity, we will assume that $\beta$ is odd ($\beta \neq 2\beta'$); therefore, $\beta_i$ are different and we can fix the class that we flip. Moreover, if we assume $H_2(L)=\Z_2$, then for $\beta=\beta_1+\beta_2$ with $\partial\beta \neq 0$, there is a unique one with $\partial\beta_i\neq 0$ and we can decide to always flip this one. Note that in this case, all boundary strata are of disc bubbling type.
\end{remark}

\begin{proposition}[{\cite[Theorem 4.9]{FO-3}}]\label{cancellation}
Let $(X,\om,\phi)$ be a symplectic manifold of real dimension six with a real structure. Suppose $L=\Fix(\phi)$ is orientable and $c_1(TX)=0$. Then $\tau_{\mf{glue}}$ is an orientation-reversing involution.
\end{proposition} 

\begin{corollary}
If $\beta$ is odd, the terms on the right-hand side of (\ref{boundary}) come in pairs with opposite signs. Therefore, the right-hand side of (\ref{boundary}) is zero, and the numbers $N^{\disc}_{\beta,J,\mf{s}}$ are independent of $J$ and of the multisection~$\mf{s}$.
\end{corollary}

If $(X,\omega,L,\phi)$ is as before and $\partial{\beta}\neq 0$, the numbers $N^{\disc}_{\beta}= N^{\disc}_{\beta,J,\mf{s}}$ defined above are called \textsf{open} Gromov-Witten invariants of $(X,L)$ in the class $\beta \in H_2(X,L)/\sim$. These numbers are independent of the choices of real almost complex structure $J$, of $\phi$-compatible Kuranishi structure, multisection $\mf{s}$, and of isotopy class of antisymplectic involution fixing $L$. 
In a similar fashion, one can define \textsf{open} GW invariants for other symplectic manifolds and with marked points.

\subsection{Relative open GW invariants}\label{relative-open-invariants}

Let $(X,\om,\phi)$ be as before and $D\subset X$ be a smooth symplectic divisor invariant under $\phi$ such that $L\cap D=\emptyset$. The definition of relative open GW invariants is a combination of the definitions of open GW invariants and of ordinary relative GW invariants. In Appendix \ref{ch:kuranishi}, we outline the construction of Kuranishi structures for the compactified relative open moduli spaces. We use these relative moduli spaces to derive a sum formula, as done in \cite{IP2} for closed GW invariants, to relate the open GW invariants of $(X,L)$ defined in Section~\ref{open-invariants-review} and the ordinary relative GW invariants of $(X_+,D)$.

\begin{definition}\label{D-compatible}
An almost complex structure $J\in \mc{J}_{\phi}$ is said to be compatible with $D$ if $J$ preserves $TD$ and
$$N_J(\xi,v)\in T_xD \qquad \forall v\in T_xD,\; \xi \in T_xX,\;x\in D,$$
where $N_J$ is the Nijenhuis tensor of $J$.
\end{definition}

A $J$-holomorphic map $u\colon (\Sigma,\partial\Sigma) \to (X,L)$ is called $D$-\textsf{regular} if it has no components mapped into $D$. If $J\in \mc{J}_{\phi}$ is $D$-compatible, a $D$-regular $J$-holomorphic map intersects $D$ in a finite set $(p_1,\cdots,p_k)$ of points  with positive multiplicities $(s_1,\cdots,s_k)$, just as in the holomorphic situation and $s_1+\cdots+s_k=[u]\cdot D$. The vector $\rho=(s_1,\cdots,s_k)$ is called the \textsf{intersection pattern}. Since $L \cap D= \emptyset$, all intersection points are interior points. In relative open GW theory, we are interested in the moduli space $\mc{M}_{k,l}^{\disc}(X,L,D,\rho,\beta)$, whose elements are $[(\Sigma,\partial\Sigma),u,\vec{z},\vec{w},\vec{\xi}]$, where
\begin{itemize}
\item $u\colon (\Sigma,\partial\Sigma) \to (X,L)$ is a $D$-regular genus zero $J$-holomorphic map,
\item $\vec{z}$ and $\vec{w}$  are tuples of $k$ boundary and $l$ interior marked points, respectively,
\item $[u]=\beta \in H_2(X,L)$ and $u^{-1}(D)=\sum s_i \xi_i$, i.e.  $\vec{\xi}$ is the set of ordered marked points corresponding to intersection points with $D$ with contact of order $s_i$ at~$u(\xi_i)$,
\end{itemize}
such that the marked map $(u,\Si,\partial\Si,\vec{w},\vec{z},\vec{\xi})$ is stable.

We next describe a suitable compactification of this moduli space, denoted by $\ov{\mc{M}}_{k,l}^{\disc}(X,L,D,\rho,\beta)$, and an orientable closed virtual cycle with which to define GW invariants.

The limiting maps in the stable compactification of this moduli space might not be $D$-regular and might have several components mapping into $D$. Since $L\cap D=\emptyset$, all the components of a limiting curve which are mapped into $D$ are maps from closed curves attached to other components away from the boundary. So the definitions of relative stable maps in \cite{IP1} and \cite{Li} readily extend to this case.

The normal bundle $\mc{N}_{D}^{X}$ of $D$ in $X$  is a complex line bundle with an inner product and a compatible connection induced by the Riemannian connection on $X$. Define 
$$Y_D= \P(\mc{N}_{D}^{X}\oplus \C).$$
The bundle map $\iota \colon \mc{N}_{D}^{X} \to Y_D$ defined by $\iota(x) = [x,1]$ on each fiber is an embedding onto the complement of the infinity section $D_{\infty} \subset Y_D$. There is a $\C^*$-action on $Y_D$ which comes from scalar multiplication on $\mc{N}_{D}^{X}$. Over each point of $D$, we can identify the fiber of $Y_D$ with $\P^1$ and give it the K\"ahler structure $(\omega_{\ep},j)$ of the 2-sphere of radius~$\ep$. Then $\iota \colon \C \to Y_D$  is a holomorphic map with $\iota^* \omega_{\ep}= d\psi_{\ep} \wedge d\theta$, where
$$ \psi_{\ep}(r) = \displaystyle\frac{2\ep^2 r^2}{1+r^2}. $$
This construction globalizes by interpreting $r$ as the norm on the fibers of $\mc{N}_{D}^{X}$, replacing $d\theta$ by the connection 1-form $\alpha$ of $\mc{N}_D^X$, and including the curvature $F$ of that connection as in Section~\ref{sec:sum}. Thus,
$$ \iota^* \omega_{\ep} = \pi ^* \omega_D + d (\psi_{\ep} \wedge \alpha)$$
is a closed form which is nondegenerate for small $\ep$ and its restriction to each fiber of $\mc{N}_{D}^{X}$ agrees with the volume form on the 2-sphere of radius $\ep$. Furthermore, at each point $p \in \mc{N}_{D}^{X}$, the connection identifies $T_p\mc{N}_{D}^{X} $ with the fiber of $\mc{N}_{D}^{X}\oplus TD$ at $\pi(p)$ and thus induces a complex structure on $Y_D$. Note that we have two copies of $D$ inside $Y_D$, corresponding to the zero section and the section at infinity, which we denote by $D_0$ and $D_{\infty}$, respectively.

Let $X[n]$ be the singular space obtained by attaching $n$ copies of $Y_D$ to $X$ in such a way that the divisor $D_0$ of the $i$-th copy is attached to the divisor $D_{\infty}$ of the $(i+1)$-th copy; see Figure~\ref{Fig:singular}. Similar to Section \ref{ch:surgery}, $X[n]$ can be realized as the singular central fiber of a symplectic fibration $\pi\colon \mc{X}[n] \to \De ^n$ whose generic fiber is a smooth symplectic manifold isotopic to $X$. 

\begin{figure}
\begin{center}
\includegraphics[scale=.5]{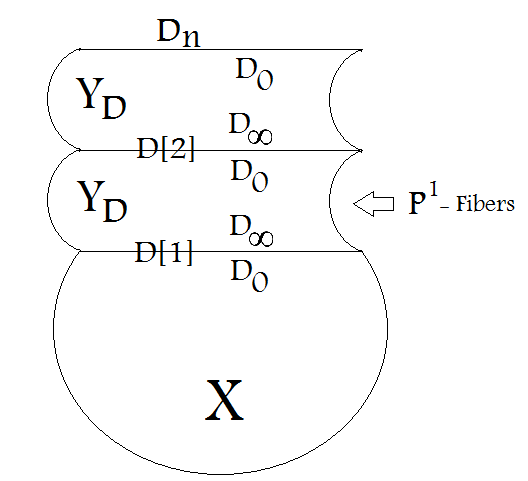}
\caption{The singular manifold $X[n]$.}\label{Fig:singular}
\end{center}
\end{figure}

Let $D_n$ be the last copy of $D_{0}$ in the sequence and $Y_D[0]=X$. For each $i=1,\cdots,n$, let $Y_D[i]$ be the $i$-th copy of $Y_D$ and $D[i]\subset Y_D[i]$ the $i$-th copy of $D_{\infty}$. The space $X[n]$ contains a copy of $L$ which lies in  $X$ itself and is disjoint from all $Y_D$. There is an action of $G_n=(\C^*)^n$ on $X[n]$ which comes from the $\C^*$ action on each copy of $Y_D$.

\begin{definition}\label{stable-relative-map}
A \textsf{stable relative} genus zero bordered $J$-holomorphic map to $X[n]$ is a tuple $[u,(\Si,\partial\Si),\vec{z},\vec{w},\vec{\xi}]$,
where 
\begin{enumerate}
\item $\Si = \Si_0\cup\cdots\cup\Si_n$ is a connected bordered nodal curve of arithmetic genus zero, $\Si_i$, $i \ge 1$, is a closed curve (not necessarily connected), and $\partial\Si\cong S^1$,
\item $\vec{z},\vec{w},\vec{\xi}$ are tuples of distinct smooth points on $\Si$, $\vec{z}$ is a tuple of boundary marked points in an anticlockwise order, $\vec{w}$ is a tuple of interior marked points, and $\vec{\xi}$ is a tuple of marked points on the last layer $\Si_n$,
\item $u\colon \Si_0 \to X$ is a $D$-regular $J$-holomorphic map into $X$ and $u\colon \Si_i \to Y_D[i]$ are $D$-regular $J_{Y_D}$-holomorphic maps,
\item $u^{-1}(D[i])= \left\{\xi_{i,1},\cdots,\xi_{i,j_i}\right\}$ is discrete for $i=1,\cdots,n$, each $\xi_{i,j}$ is a node of $\Si$ connecting $\Si_{i-1}$ and $\Si_{i}$, and $u\mid_{\Si_{i-1}}$ and $u\mid_{\Si_{i}}$ have same contact orders with $D[i]$,
\end{enumerate}
such that the automorphism group of $f=[u,(\Si,\partial\Si),\vec{w},\vec{z},\vec{\xi}],$ 
$$ \mf{Aut}(f)= \left\{(h,\si)\mid \si \in G_n,\; h \in \mf{Aut}(\Si,\vec{z},\vec{w},\vec{\xi}), \;\;\si\circ u=u\circ h\right\},$$ 
is finite.

\end{definition}

Let $\ov{\mc{M}}_{k,l}^{\disc}(X,L,D,\rho,\beta)$ denote the set of the equivalence classes of all bordered stable relative maps in class $\beta$ and with intersection pattern $\rho$. This moduli space is Hausdorff and compact. In Appendix \ref{ch:kuranishi}, we outline a construction of a virtual fundamental class for this space.

Let
$$\ev=(\ev^B_{\vec{z}},\ev_{\vec{w}},\ev_{\vec{\xi}}) \colon \ov{\mc{M}}_{k,l}^{\disc}(X,L,D,\rho,\beta)\to L^{k} \times X^{l} \times D^{l(\rho)},$$ 
where $l(\rho)= m$ if $\rho=(s_1,\cdots,s_m)$, be the total evaluation map. \textsf{Open relative GW invariants} are obtained by integrating pull-backs under $\ev$ of differential forms on $\ov{\mc{M}}_{k,l}^{\disc}(X,L,D,\rho,\beta)$. As in the absolute case, we encounter issues concerning codimension one boundaries and orientation. If $L$ is the fixed point set of an antisymplectic involution preserving $D$, we can use the same technique as in Section \ref{open-invariants-review} to define relative open invariants. The antisymplectic involution on $X$ extends to $X[n]$ and induces an involution on $\ov{\mc{M}}_{k,l}^{\disc}(X,L,D,\rho,\beta)$. Therefore, we can get a cancellation of boundary terms as in Proposition \ref{cancellation}.

\begin{remark}\label{Gamma}
We also need to consider relative invariants with disconnected domains,
$$(\Si,\partial\Si)= (\Si_0,\partial\Si_0) \cup \Si_1 \cup \cdots \cup \Si_k,$$ 
where $\Si_i$, $ 1\leq i \leq k$, has no boundary and $\partial\Si_0 \cong S^1$. We fix homology classes $\beta_0 \in H_2(X,L)/\sim$ and $\beta_i \in H_2(X)/\sim$, for $ 1\leq i \leq k$, with $ \beta=\beta_0+\sum \beta_i$. With $\Gamma$ denoting the above topological data, let $\ov{\mc{M}}_{k,l}^{\disc}(X,L,D,\rho,\Gamma)$ be the moduli space  of relative maps $u \colon (\Si,\partial\Si)\to (X[n],D_n)$ so that over each component $u$ has the given topological type. 
\end{remark}

\begin{example}\label{example1}
Let $(X,\omega,\phi,L,D)=(\C\P^3,\omega_{\tn{FS}},\tau_3,\R\P^3,Q)$, where $Q=Q^2$ is the quadratic hypersurface with the real defining equation
$$ x_0^2+x_1^2+x_2^2+x_3^2=0.$$
Since $H_2(\P^3, \R\P^3)\cong\frac{1}{2}\Z$, we write the elements of $H_2(X,L)$ by $[\frac{d}{2}]$. Note that $\mu([\frac{d}{2}])= 4d \equiv 0 \; \tn{mod}\; 4$. Let $\rho_0=(1,\cdots,1)$, $d$ be odd, and 
$$\ev_{\xi}\colon \ov{\mc{M}}^{\disc}_{0,0}(\P^3,\R\P^3,Q,\rho_0,[\frac{d}{2}])\to Q^d$$
be the evaluation map at the contact points. Consider the incidence condition 
$$\gamma:= \left\{p_1,\tau_3(p_1)\right\}\times \cdots \times \left\{p_d,\tau_3(p_d)\right\} \subset Q^d,$$
where $p_1,\cdots,p_d$ are $d$ general points in $Q$.  We choose $\gamma$ symmetric with respect to the involution in order to make $\ov{\mc{M}}^{\disc}_{0,0}(\P^3,\R\P^3,Q,\rho_0,[\frac{d}{2}])\times_{\ev} \gamma$ closed under the involution $\tau_{\mf{glue}}$. Then $\ov{\mc{M}}^{\disc}(\P^3,\R\P^3,Q,\rho_0,[\frac{d}{2}])\times_\ev\!\gamma$ has virtual dimension zero and virtually counts the number of holomorphic discs intersecting $Q$ at certain fixed points determined by $\gamma$. For $d$ odd, the rational number
$$
\alpha_d^{\tn{rel},\disc}= \frac{1}{2^d}\# [\ov{\mc{M}}^{\disc}_{0,0}(\P^3,\R\P^3,Q,\rho_0,[\frac{d}{2}])\times_{\ev} \gamma]^{\vir}
$$
is well-defined; for $d$ even, we need extra terms to make it invariant. Let $N_d^{\disc}$ be the ordinary open GW invariants defined by
\begin{equation}\label{disc-inv}
N_d^{\disc}= \frac{1}{2^d}\# [\ov{\mc{M}}^{\disc}_{0,d}(\P^3,\R\P^3,[\frac{d}{2}])\times_{\ev} \gamma]^{\vir}.
\end{equation}
Since $\P^3$ is Fano and $\rho_0=(1,\cdots,1)$, $\alpha_d^{\tn{rel},\disc} = N_{d}^{\disc}$,
where $N_d^{\disc}$ counts the number of degree $d$ disks in $\P^3$ passing through $d$ pairs of conjugate points.

For a non-connected domain $\Si=(\Si_0,\partial{\Si_0}) \cup \bigcup_{i=1}^{k} \Si_{k}$, let $\Gamma$ be a topological type as before with $\beta_0 \in H_2(\C\P^3,\R\P^3)$ and $\beta_{i} \in H_2(\C\P^3)$, for $i=1,\cdots,k$. Let $d_i=\mu(\beta_i)/4 \in \Z$ and  
\begin{equation}\label{incidence}
\gamma_{\Gamma}= \gamma_0 \times \gamma_{1} \times \cdots \times \gamma_{k},
\end{equation}
where $\gamma_0=\gamma$ as before and $\gamma_i=\left\{q_{i1}\right\}\times \cdots \times \left\{q_{id_i}\right\}$  is a single point in $Q^{d_i}$, whenever $i\geq 1$. We define
\begin{equation}\label{alpha-invariants}
\alpha_{\Gamma}^{\tn{rel},\disc}=\frac{1}{2^{d_0}} \# [\ov{\mc{M}}^{\disc}_{0,0}(\P^3,\R\P^3,Q,\rho_0,\Gamma) \times_{\ev} \gamma_{\Gamma}]^{\vir}.
\end{equation}
These numbers are invariant under deformations of almost complex structure and $\gamma_i$'s and are generalizations of the above invariants to non-connected domains. Since
$$\ov{\mc{M}}^{\disc}_{0,0}(\P^3,\R\P^3,Q,\rho_0,\Gamma)=\ov{\mc{M}}^{\disc}_{0,0} (\P^3,\R\P^3,Q,\rho_0,[\frac{d_0}{2}]) \times \prod_{i=1}^{k} \ov{\mc{M}}_{0,0}(\P^3,Q,\rho_0,[d_i]),$$
we find that
$$\alpha_{\Gamma}^{\tn{rel},\disc}=\alpha^{\tn{rel},\disc}_{d_0} \times \prod_{i=1}^{k} \alpha_{d_i}^{\tn{rel}}.$$
This shows that disconnected invariants reduce to connected ones. Again for dimensional reason, $\alpha_{\Gamma}^{\tn{rel},\disc}$ are equal to absolute open invariants $N_{\Gamma}^{\disc}$ defined by an equation similar to (\ref{disc-inv}). 

\end{example}

\section{Degeneration of moduli spaces\label{ch:degeneration}}

In this section, we build a cobordism between moduli spaces of holomorphic discs in a smooth fiber and the fiber product of moduli spaces of relative maps in the singular fiber in the fibration $\pi:\mc{X}\to \De$ constructed in Section \ref{ch:surgery}. Using this cobordism, we prove Theorems \ref{emptiness} and \ref{open-closed}.

\subsection{Proof of Theorem \ref{emptiness}}

Given $(X,\om,L)$ as in the statement of Theorem \ref{emptiness}, let $\pi\colon \mc{X} \to \Delta$ be the associated fibration constructed in Section \ref{ch:surgery}. 
Let $\mc{J}_{\mc{X}}^l$ be the set of compatible almost complex structures $J$ of class $C^l$ on $\mc{X}$, given by Proposition \ref{prop:main-space}. Restricted to $X_+$, any such $J$ is $D$-compatible in the sense of Definition \ref{D-compatible}. Let $\mc{M}^{\tn{reg}}(X_{+},A)$ be the moduli space of degree $A$ genus zero somewhere injective $J|_{X_+}$-holomorphic curves.
By \cite[Theorem 3.1.5]{MS2} and Proposition \ref{prop:main-space},
$$\dim_{\R}^{\vir}(\mc{M}^{\tn{reg}}(X_{+},A)) = 2(n-3) + -2(n-2)~[A]\cdot[D],$$
which is a negative number if $[A]\cdot [D] > 0$ and $n>2$.

\begin{lemma}\label{empty}
There is a dense subset $\mc{J}_{\mc{X}}^{l,\tn{reg}} \subset \mc{J}_{\mc{X}}^l$ such that $\mc{M}^{\tn{reg}}(X_{+},A)=\emptyset$ for every $J \in \mc{J}_{\mc{X}}^{l,\tn{reg}}$ and every $A \in H_2(X_{+},\Z)$ with $[A]\cdot [D] > 0$.
\end{lemma}

\begin{proof}
Let 
$$\mc{M}^{\tn{reg}}(X_{+},\mc{J}_{\mc{X}}^l,A)= \coprod_{J\in \mc{J}_{\mc{X}}^l} \mc{M}^{\tn{reg}}(X_{+},J,A)$$
be the universal moduli space. By \cite[Chapter 6]{MS2}, the linearization map
\begin{equation*}
\begin{split}
D_{J,u} \colon W^{k,p}(\P^1, u^*TX_{+}) \oplus T_J\mc{J}_{\mc{X}}^l &\to W^{k-1,p} (\P^1,u^*TX_{+} \otimes_J \La_J^{0,1} T\P^1),\\
D_{J,u} (\xi,Y) &= L_{J,u}(\xi) + \frac{1}{2} Y(u) du\circ j,
\end{split}
\end{equation*}
of the Cauchy-Riemann operator is surjective for every $(u,J) \in \mc{M}^{\tn{reg}}(X_{+},\mc{J}_{\mc{X}}^l,A)$. 
The projection map $\pi \colon \mc{M}^{\tn{reg}}(X_{+},\mc{J}_{\mc{X}}^l,A) \to \mc{J}_{\mc{X}}^l$ is Fredholm, and the kernel and cokernel of $d\pi$ are isomorphic to the kernel and cokernel of $L_{J,u}$. By the Sard-Smale theorem \cite[Theorem A.5.1]{MS2}, the set of regular values of $\pi$ is of the second category, provided $l-1 \geq 0, \index(L_{J,u})$. On the other hand, $J \in \mc{J}_{\mc{X}}^l$ being a regular value for $d\pi$ means that $L_{J,u}$ is surjective, and so $\mc{M}^{\tn{reg}}(X_{+},J,A)$ is a negative-dimensional smooth moduli space and therefore empty. Taking intersection over all curve classes $A \in H_2(X_{+},\Z)$, we find a dense set of almost complex structures for which all the moduli spaces $\mc{M}^{\tn{reg}}(X_{+},J,A)$ with $[A]\cdot[D]> 0$ are empty. 
\end{proof}

Given $J\in \mc{J}^{l,\tn{reg}}_{\mc{X}}$ and $\la \in \De$, let $J_{\la}=J|_{X_{\la}}$ as before. Suppose $E>0$, $\la_{i}\in \De^*$ is a sequence converging to $0$, and $[u_i]$ is a sequence of $J_{\la_i}$-holomorphic discs such that $\om_{\mc{X}}([u_i])<E$. This sequence is a sequence of $J$-holomorphic discs in a compact subset of $\mc{X}$ with a uniform energy bound. By the Gromov Compactness Theorem, there is a $J_0$-holomorphic map $u_0\colon B \to \mc{X}_0$ and a sequence of orientation-preserving diffeomorphisms $\psi$ of the domains of $u_i$ such that a subsequence of $u_i\circ \psi_i$ converges to $u_0$.
Furthermore, every component of $u_0$ has image in either $X_{-}$ or $X_{+}$ and least one component maps to $X_{+}$ intersecting $D$ in a nonempty discrete set. The last claim holds for the following reason. Each $u_{\la_i}$ intersects $L$, so $u_0$ has non-empty intersection with $L \subset X_{-}$, which means there is an irreducible component $u_{-}$ of $u_0$ mapped into $X_{-}$. We know $D$ is obtained by a symplectic cut along some hypersurface $V_a=h^{-1}(a) \subset X$. Consider the contact hypersurfaces $V_{a+\ep}$ (for $\ep > 0$ small) in $X$. After performing symplectic cut along $V_a$, we get  copies of $ V_{a+\ep}$ in $X_{+}$ which are the boundaries of a tubular neighborhood of $D \subset X_{+}$. Each $u_{i}$ has a non-empty intersection with $V_{a+\ep}$, because the symplectic form inside the neighborhood of $L$ surrounded by $V_{a+\ep}$ is exact and so there are no $J_{\lambda_i}$-holomorphic disc completely inside $V_{a+\ep}$. Thus, the limit curve $u_0$ has a non-empty intersection with $V_{a+\ep} \subset X_{+}$, and so there are some irreducible components of $u_0$ mapped into~$X_{+}$  and not contained in $D \subset X_{+}$. Since the domain of $u_0$ is connected, a component of $u_0$ mapped into $X_+$ and not contained in $D$ intersects $D$. Since $J_0|_{X_+}$ is $D$-compatible, this component intersects $D$ at finitely many points. 
However, by Lemma \ref{empty}, $\mc{M}^{\tn{reg}}(X_+,J,A)=\emptyset$ if $J\in \mc{J}_{\mc{X}}^{l,\tn{reg}}$ and $A\cdot D>0$. Thus, $\ov{\mc{M}}^{\disc}(X,L,J_{\la_i},\beta)=\emptyset$ for all $\la\in \De^*$ small and $\beta\in H_2(X,L)$ such that $\om(\beta)<E$. Once again, by the Gromov Compactness Theorem this also holds for some neighborhood $U_E$ of $J_{\la} \in \mc{J}_{X_{\la}}$. Since $X\cong \mc{X}_{\la}$, this finishes the proof of Theorem \ref{emptiness}.

\subsection{Proof of Theorem \ref{open-closed}}\label{proofs}

Let $(X,\om,\phi)$ be a symplectic manifold with a real structure such that $c_1(TX)=0$ and $L=\Fix(\phi)=\R\P^3$. Let $\pi:\mc{X}\to \De$ be the associated fibration of Section~\ref{ch:surgery} and $\mc{Y}=\pi^{-1}([0,1]) \subset \mc{X}$. Each fiber of $\mc{Y}\to[0,1]$ is invariant under the induced involution $\phi_{\mc{X}}$. Fix some compatible $J$ on $\mc{X}$ and define 
$$\ov{\mc{M}}^{\disc}(\mc{Y},L,\left\{J_t\right\}_{t\in (0,1]},\beta)= \bigcup_{t\in(0,1]} \ov{\mc{M}}^{\disc}(\mc{X}_t,L,J_t,\beta).$$ 
Let $\ov{\mc{M}}^{\disc}(\mc{Y},L,\{J_t\}_{t\in I},\beta)$, where $I=[0,1]$, be the relative stable map compactification of $\ov{\mc{M}}^{\disc}(\mc{Y},L,\{J_t\}_{t\in (0,1]},\beta)$, similar to \cite{IP2}, \cite{Li}, and Section \ref{relative-open-invariants}, including maps to the fiber over zero.

Every element $(u,\Si)$ of $\ov{\mc{M}}^{\disc}(\mc{Y},L,\{J_t\}_{t\in I},\beta)$ in $\mc{X}_0$ belongs to a fiber product of  relative moduli spaces over $X_{-}$ and $X_{+}$ with matching conditions on $D$,
\begin{equation}\label{equ:fiber-prod}
(u,\Sigma) \in  \ov{\mc{M}}^{\disc}(X_-,L,D,\rho,\Gamma_{-})\times_{(\ev_{\xi^-},\ev_{\xi^+})}\ov{\mc{M}}(X_+,D,\rho,\Gamma_{+}),
\end{equation}
where $\ov{\mc{M}}^{\disc}(X_-,L,D,\rho,\Gamma_{-})$ and $\ov{\mc{M}}(X_+,D,\rho,\Gamma_{+})$ are the relative moduli spaces with the same intersection pattern $\rho$, $\xi^\pm$ are contact points with $D$, and $\Gamma_{\pm}$ encodes the data corresponding to the topological types of the domain and image.

There is a fiber-wise involution $\tau_{\mc{M}}$ on $\ov{\mc{M}}^{\disc}(\mc{Y},L,\{J_t\}_{t\in [0,1]},\beta)$ as before. Fix a spin structure $\sigma$ on $L$. For a tuple $\rho=(s_1,\cdots,s_k)$, let $|\rho|=\prod s_i$.

\begin{proposition}\label{kur-str-main}
Let $(X^6,\omega,\phi)$ be a symplectic manifold with $c_1(TX)=0$. The moduli space $\ov{\mc{M}}^{\disc}(\mc{Y},L,\{J_t\}_{t\in I},\beta)$ has a topology with respect to which it is compact and Hausdorff. It has a $\tau_{\mc{M}}$-invariant oriented Kuranishi structure of virtual dimension $1$ with respect to which the projection $\pi$ is smooth. The codimension one boundary components correspond to the moduli spaces of the form (\ref{boundary-equ1}), (\ref{boundary-equ2}), the fiber over $\la=1$, and a covering of (\ref{equ:fiber-prod}), \\ 
\fontsize{10}{12}\selectfont
\xymatrix{
&\mc{M}(\mc{X}_0,\rho,\Gamma_{-},\Gamma_{+}) \ar[d]_{|\rho|-\tn{covering}}^{\pi} \ar[r]^{\iota_{(\Gamma_{-},\Gamma_{+})}} & \partial\ov{\mc{M}}^{\disc}(\mc{Y},L,\{J_t\}_{t\in I},\beta)\\
&\mc{M}^{\disc}(X_-,L,D,\rho,\Gamma_{-})\times_{(\ev_{\xi^-},\ev_{\xi^+})}\mc{M}(X_+,D,\rho,\Gamma_{+}) &}
\vskip.05in
\normalsize
\noindent which is compatible with the Kuranishi structures and the orientation induced by the spin structure $\sigma$.
\end{proposition}

In Appendix \ref{ch:kuranishi} below, we describe the Kuranishi structure and the covering space $\mc{M}(\mc{X}_0,\rho,\Gamma_-,\Gamma_+)$. 

We now turn to the proof of Theorem \ref{open-closed}. By Proposition \ref{kur-str-main},
\begin{equation*}
\begin{split}
&N^{\disc}_{(\beta,J\mid_{\pi^{-1}(1)},\mf{s}\mid_{\pi^{-1}(1)})} = \sum_{(\Gamma_{-},\Gamma_{+}),\rho} \displaystyle\frac{1}{\mf{Aut}(\Gamma_{-},\Gamma_{+})}\; \# [\mc{M}(\mc{X}_0,\rho,\Gamma_-,\Gamma_+)]^{\mf{s}} \\
&\qquad - \sum_{\beta_1+\beta_2=\beta}\# [\ov{\mc{M}}^{\disc}_{1,0}(\mc{Y},L,\{J_t\}_{t\in \overset{\circ}{I}},\beta_1)\!\times_{(\ev_1,\ev_1)}\! \ov{\mc{M}}^{\disc}_{1,0}(\mc{Y},L,\{J_t\}_{t\in \overset{\circ}{I}},\beta_2)]^{\mf{s}},
\end{split}
\end{equation*}
where $\overset{\circ}{I}=(0,1)$ and $\mf{Aut}(\Gamma_{-},\Gamma_{+})$ is the finite automorphism  group of the $(\Gamma_{-},\Gamma_{+})$ configuration. By Proposition \ref{cancellation}, the last term above is zero. Therefore, 
\begin{equation}\label{equ:summation}
N^{\disc}_{\beta} = \sum_{(\Gamma_{-},\Gamma_{+}),\rho} \displaystyle\frac{1}{\mf{Aut}(\Gamma_{-},\Gamma_{+})}\; \# [\mc{M}(\mc{X}_0,\rho,\Gamma_-,\Gamma_+)]^{\mf{s}}.
\end{equation}
The sum on the left-hand side of (\ref{equ:summation}) corresponds to the boundary terms coming from the central fiber, in a similar way to \cite[Theorem 3.15]{Li} and \cite[Theorem 12.3]{IP2}. Therefore, (\ref{equ:summation}) is an open version of the symplectic sum formula.

If $\rho=(s_1,\cdots,s_k)$,
$$\dim^{\tn{vir}}(\ov{\mc{M}}(X_+,D,\rho,\Gamma_{+}))= k - \sum_{i=1}^k s_i.$$ 
Therefore, the only $\rho$ for which we get a non-trivial contribution is the trivial one, $\rho_0=(1,\cdots,1)$. By Proposition \ref{kur-str-main}
$$\mc{M}(X_0,\rho_0,\Gamma_{-},\Gamma_{+})=\ov{\mc{M}}^{\disc}(X_{-},L,D,\rho_0,\Gamma_{-})\times_{(\ev_{\xi^-},\ev_{\xi^+})}\ov{\mc{M}}(X_{+},D,\rho_0,\Gamma_{+}).$$
It remains to understand the fiber-product term on the right. Since $\ov{\mc{M}}(X_+,D,\rho_0,\Gamma_{+})$ has virtual dimension zero and $\tau_{\mc{M}}$-invariant Kuranishi structure,
$$\ev_{\xi^{+}}[\ov{\mc{M}}(X_+,D,\rho_0,\Gamma_{+})]^{\mf{s}} \subset D^{k}$$
is a $\phi_{X_{+}}$-invariant zero-dimensional chain, which we denote by $\gamma_{\Gamma{+}}$. Then $N^{\tn{rel}}_{\Gamma_{+}}=\left|\gamma_{\Gamma_{+}}\right| \in \Q$ is the closed relative GW invariants of the class $\Gamma_+$ counting elements of the corresponding relative moduli space. Therefore,
\begin{equation*}
\begin{split}
[\ov{\mc{M}}^{\disc}(X_-,L,D,\rho_0,\Gamma_{-})\times_{(\ev_{\xi^-},\ev_{\xi^+})}\ov{\mc{M}}(X_+,D,\rho_0,\Gamma_{+})]^{\mf{s}}&\\ =[\ov{\mc{M}}^{\disc}(X_-,L,D,\rho_0,\Gamma_{-})\times_{\ev_{\xi^-}} \gamma_{\Gamma_{+}}]^{\mf{s}}&.
\end{split}
\end{equation*}

\begin{lemma}\label{lem:decomposition}
With the notation as above,
$$\# [\ov{\mc{M}}^{\disc}(X_-,L,D,\rho_0,\Gamma_{-})\times_{\ev_{\xi^-}} \gamma_{\Gamma_{+}}]^{\mf{s}}= \alpha_{\Gamma_{-}}^{\tn{rel},\disc} N_{\Gamma_{+}}^{\tn{rel}} ,$$
where $\alpha_{\Gamma_{-}}^{\tn{rel},\disc}$ are the numbers defined in (\ref{alpha-invariants}).
\end{lemma}

\begin{figure}
\begin{center}
\includegraphics[scale=.6]{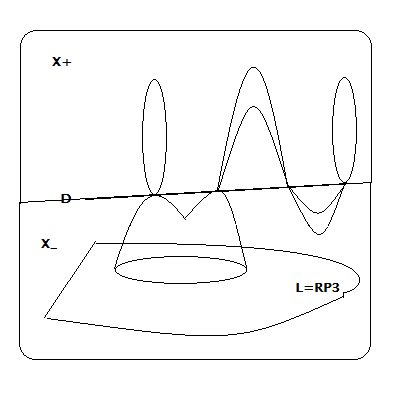}
\end{center}
\caption{A typical $J$-holomorphic map in singular fiber.}\label{Fig:relative-curve}
\end{figure}

\begin{proof}
A typical element of 
$$\ov{\mc{M}}^{\disc}(X_-,L,D,\rho_0,\Gamma_{-})\times_{(\ev_{\xi^-},\ev_{\xi^+})}\ov{\mc{M}}(X_+,D,\rho_0,\Gamma_{+})$$
represents a curve as in Figure~\ref{Fig:relative-curve}.
In this figure, the domain of $\Gamma_{+}$ consists of three rational curves, while the domain of $\Gamma_{-}$ consists of a disc component and a rational curve. Since $\Gamma_{-}\#\Gamma_{+}$ is a degeneration of the disc, the unique bordered component of $\Gamma_{-}$, say $(\Sigma_0,\partial{\Si}_0)$, intersects the domain of $\Gamma_{+}$ in disjoint irreducible components. Thus, if the topological type of $\Gamma_{-}$ over $(\Si_0,\partial{\Si}_0)$ is $[\frac{d_0}{2}]\in H_2(\C\P^3,\R\P^3)$, then $\Gamma_{+}$ has at least $d_0$ components.

Let $(\xi_1,\cdots,\xi_{d_0})$ denote the intersection points of $\Si_0$ with the common divisor $D$ and let $(\xi_1,\cdots,\xi_{k}) \in D^k$ be the set of all intersection points. For $\ep=(\ep_1,\cdots,\ep_{l}) \in \Z_2^{l}$, let $\mf{c}_{\ep}\colon D^{l} \to D^{l}$ be the map which is equal to the identity on $i$-th factor if $\ep_i=0$ and  to $\tau_3$ if $\ep_i=1$. 

If $(\xi_1,\cdots,\xi_k) \in \gamma_{\Gamma_{+}}$, then $(\mf{c}_{\ep}(\xi_1,\cdots,\xi_{d_0}),\mf{c}_{\ep'}(\xi_{d_0+1},\cdots,\xi_{k}))$ is also in $\gamma_{\Gamma_{+}}$, where $\ep\in \Z_2^{d-0}$ is arbitrary and $\ep'\in \Z_{2}^{k-d_0}$ depends on $\Gamma_{+}$ and $\ep$. This is because $\Si_0$ meets $\Gamma_{+}$ at disjoint components $\Si_i$, and for every $J_{+}$-holomorphic map $u\colon \Sigma_{i} \to X_{+}$, $\tau_{\mc{M}}(u)$ is also a $J_{+}$-holomorphic map with same the topological type; that is,  we can flip each individual component of $\Gamma_{+}$ using the induced involution $\phi_{+}$ on $X_{+}$.

Given $(\xi_1,\cdots,\xi_{k}) \in \gamma_{\Gamma_{+}}$, let $S$ be the set of the $2^{d_0}$ tuples obtained from $(\xi_1,\cdots,\xi_{k})$. Let $q=(q_{d_0+1},\cdots,q_{k}) \in D^{k-d_0}$. For each point $\mf{c}_{\ep'}(\xi_{d_0+1},\cdots,\xi_{k})$ as above, choose a path $\gamma_{\ep'}(t)$, $0 \leq t \leq 1$, in $D^{(k-d_0)}$ connecting these two points. Let $S_{t}$ be the set obtained by replacing a point of the form $(\mf{c}_{\ep}(\xi_1,\cdots,\xi_{d_0}),\mf{c}_{\ep'}(\xi_{d_0+1},\cdot,\xi_{k}))$ with 
$$(\mf{c}_{\ep}(\xi_1,\cdots,\xi_{d_0}),\gamma_{\ep'}(t)).$$ 
Then, $S_1$ is of the form $\gamma_0\times \left\{q\right\} \in D^{k}$ as in (\ref{incidence}).
For each $S_t$, define 
$$\mc{M}^{\disc}_{t,\Gamma_{-}}=\ov{\mc{M}}^{\disc}(X_-,L,D,\rho_0,\Gamma_{-})\times_{\ev_{\xi^-}} S_t$$
to be the zero-virtual-dimensional relative moduli space with incidence condition determined by $S_t$ and 
$$\mc{M}^{\disc}_{\Gamma_{-}}= \bigcup_{t\in I} \mc{M}^{\disc}_{t,\Gamma_{-}}$$
to be their union. Then $[\mc{M}_{\Gamma_{-}}^{\disc}]^{\mf{s}}$ is a one-dimensional cobordism between $[\mc{M}^{\disc}_{0,\Gamma_{-}}]^{\mf{s}}$ and $[\mc{M}^{\disc}_{1,\Gamma_{-}}]^{\mf{s}}$, because the part of the incidence condition which corresponds to the disc part $(\Sigma_1,\partial{\Sigma_1})$ of $\Gamma_{-}$ is fixed and the disc part of each curve in cobordism is fixed, and so disc-bubbling does not happen in the middle. We conclude that
\begin{equation*}
\begin{split}
\#[\ov{\mc{M}}^{\disc}(X_-,L,D,\rho_0,\Gamma_{-})\times_{\ev_{\xi^-}} S_0]^{\mf{s}}& =\# [\ov{\mc{M}}^{\disc}(X_-,L,D,\rho_0,\Gamma_{-})\times_{\ev_{\xi^-}} S_1]^{\mf{s}}\\
& = 2^{d_0} \alpha_{\Gamma_{-}}^{\tn{rel},\disc}.
\end{split}
\end{equation*}
Performing this for all points in the 0-chains $\gamma_{\Gamma_{+}}$ and then adding up all the terms gives the desired result.
\end{proof}

By $(\ref{equ:summation})$ and Lemma $\ref{lem:decomposition}$,
\begin{equation}\label{sum-formula-2}
N_{\beta}^{\disc}= \sum_{\Gamma_{-}\#\Gamma_{+}=\beta} \displaystyle\frac{1}{\mf{Aut}(\Gamma_{-},\Gamma_{+})} N_{\Gamma_{+}}^{\tn{rel}} \alpha_{\Gamma_{-}}^{\tn{rel},\disc}.
\end{equation}  
Thus, the open invariants of $(X,L)$ can be expressed as a linear function of the relative GW invariants of the symplectic manifold $X_{+}$ and the universal constants $\alpha_{\Gamma_{-}}^{\tn{rel},\disc}$. As mentioned at the end of Section~\ref{relative-open-invariants}, by dimensional reason $$\alpha_{\Gamma_{-}}^{\tn{rel},\disc}=N_{\Gamma_-}^{\disc}.$$ 

Since the intersection pattern is trivial and $N_{\Gamma_+}^{\tn{rel}}$ involves no absolute constraints (lying in $D$), $N_{\Gamma_+}^{\tn{rel}}$ is also equal to its absolute version $N_{\Gamma_+}$. Thus
$$ N_{\beta}^{\disc}(X,L)=\sum_{\beta=\Gamma_1\#\Gamma_2} \frac{1}{\mf{Aut}(\Gamma_1,\Gamma_2)} N_{\Gamma_-}^{\disc}(\P^3,\R\P^3) N_{\Gamma_+}(X_+).$$

\appendix
\section{Kuranishi structure on $\ov{\mc{M}}^{\disc}(X,L,D,\rho,\beta)$}\label{ch:kuranishi}

This section outlines a construction of Kuranishi structure on the relative moduli space $\ov{\mc{M}}^{\disc}(X,L,D,\rho,\beta)$.
It includes all the steps needed to put Kuranishi structure on the other moduli spaces in this paper. The case with marked points can be treated similarly. We refer to \cite{FOOO} and \cite{S} for the definition and basic properties of Kuranishi structures and to \cite{MS2, IP2} for the details on gluing theorems that we use here. Throughout this section $(X,\om)$ denotes a symplectic manifold, $L\subset X$ and $D\subset X$ a Lagrangian submanifold and a symplectic hypersurface, $\beta \in H_2(X,L)/\sim$, $\rho=(s_1,\cdots,s_m)$, $l(\rho)=m$. We fix a compatible almost complex structure $J$.

\subsection{Irreducible $D$-regular maps}

Let $u\colon(D^2,S^1)\to(X,L)$ be a $D$-regular  $J$-holomorphic map; thus, $ \tn{Im}(u) \cap D$ is finite. We denote the set of such maps by $\mc{M}^{*,\tn{reg},\disc}(X,L,D,\rho,\beta)$. There are $l(\rho)$ marked points $\vec{\xi}$ corresponding to the contact points with $D$. Define $$E_u=u^*TX,\qquad F_u=u|_{S^1}^*TL,\quad \tn{and}\quad E_{u}^{0,1}=(T^*D^2)^{0,1} \otimes_{\C} E_u.$$
There are commutative diagrams
\begin{equation}\begin{split}\label{diag:E}
\xymatrix{
& E_u \ar[d]_{\pi} \ar[rr]^{T_{\phi}} && E_{\tilde{u}} \ar[d]_{\pi}&&&  E_u^{0,1} \ar[d]_{\pi} \ar[rr]^{T^1_{\phi}} && E_{\tilde{u}}^{0,1} \ar[d]_{\pi}\\
& D^2             \ar[rr]^{c}     && D^2                    &&& D^2                     \ar[rr]^{\id}    && D^2}
\end{split}\end{equation}
where $c(z)=\bar{z}$, $\tilde{u}=\phi,\circ u\circ c$, $T_{\phi}v=d\phi(v)$, and $T^1_{\phi}\alpha=d\phi\circ\alpha\circ dc$.
Fix $p>2$ and $l>\max{s_i}$. Let $W^{l,p}(E_u,F_u)_{\rho}$ be the set of vector fields of class $W^{l,p}$ vanishing to order $s_i$ at each intersection point $\xi_i$ with $D$ and tangent to $TL$ along~$S^1$. Similarly, let $W^{l-1,p}(E_u^{0,1})_{\rho}$ be the set of $E_u$-valued $(0,1)$-forms of class $W^{l-1,p}$ vanishing to order $s_i-1$ at each~$\xi_i$. 
The linearized Cauchy-Riemann operator then is a map 
$$ D_u\colon W^{l,p}(E_u,F_u)_{\rho} \to W^{l-1,p}(E_u^{0,1})_{\rho}.$$ 
This fit into a commutative diagram
$$\xymatrix{
& W^{l,p}(E_u,F_u)_{\rho} \ar[d]_{\tilde{T}_{\phi}} \ar[rr]^{D_u} && W^{l-1,p}(E_u^{0,1})_{\rho} \ar[d]_{\tilde{T}^1_{\phi}} \\
& W^{l,p}(E_{\tilde{u}},F_{\tilde{u}})_{\rho}                       \ar[rr]^{D_{\tilde{u}}} && W^{l-1,p}(E_{\tilde{u}}^{0,1})_{\rho} }         $$
where $\{\tilde{T}_{\phi}\xi\}(z)=T_{\phi}(\xi(c(z)))$  and $\{\tilde{T}^1_{\phi}\alpha\}(z) = T^1_{\phi}(\alpha(z))$. 
Choose finite-dimensional subspaces $\mc{E}_u \subset W^{l-1,p}(E_u^{0,1})_{\rho}$ and $\mc{E}_{\tilde{u}} \subset W^{l-1,p}(E_{\tilde{u}}^{0,1})_{\rho}$ such that every $\eta \in \mc{E}_u $ is smooth and supported away from the boundary and marked points, $D_u$ modulo $\mc{E}_u$ is surjective, and  $\tilde{T}^1_{\phi}(\mc{E}_u)= \mc{E}_{\tilde{u}}$. The last condition guarantees that $\tau_{\mc{M}}$ induces an involution on the Kuranishi structure. 

We take our Kuranishi neighborhood to be $V(u)= (\pi\circ D_u)^{-1}(0)$, modulo the automorphism group  $\tn{PSL}(2,\R)$ of the disc, which is a smooth manifold of dimension 
$$\mu(\beta)+n-3+2l(\rho)+ \dim(E_{u})- 2D\cdot\beta.$$
The obstruction bundle at each $f \in V_u$ is obtained by parallel translation of $\mc{E}_{u}$ with respect to the induced metric of $J$. Thus, we get a vector bundle $E(u)$ and a Kuranishi neighborhood $(V(u),E(u))$. The Kuranishi map in this case is just the Cauchy-Riemann operator $f \ra \bar\partial f$. If there are additional boundary or interior marked points, the tangent space is bigger and includes the tangent spaces of marked points. In this case the Kuranishi structure is a product of Kuranishi structure of the map and the moduli space of marked points.

\subsection{Nodal $D$-regular maps}

Let 
$$[u,\Si] \in \ov{\mc{M}}^{\tn{reg},\disc}(X,L,D,\rho,\beta)\setminus \mc{M}^{*,\tn{reg},\disc}(X,L,D,\rho,\beta)$$
be so that the domain is nodal, but the image is still $D$-regular. Write $ \Si = \coprod \Si_i$, where each $\Si_i$ is a smooth curve isomorphic to either the disc or the sphere. Then each $u_i= u\mid_{\Si_i}$ is an irreducible map. For simplicity we assume there are only two components $\Si_1$, $\Si_2$ in the decomposition;  we can further assume that the two components are discs with a boundary point in common. 
The cases with more components or with sphere components can be treated similarly.

We can assume that the node is given by $1\in \Si_i$ for each of the two discs. Let $q=u_i(1) \in L$. If $[u_1]=\beta_1$ and $[u_2]=\beta_2$, then 
\begin{gather*}
(u_i,\Si_i,1) \in \mc{M}^{*,\tn{reg},\disc}_{1,0}(X,L,D,\rho_i,\beta_i), \\
[u,\Si]  \in \mc{M}^{*,\tn{reg},\disc}_{1,0}(X,L,D,\rho_1,\beta_1)\times_{(\ev_1,\ev_1)} \mc{M}^{*,\tn{reg},\disc}_{1,0}(X,L,D,\rho_2,\beta_2).
\end{gather*}
Let $(V(u_i),E(u_i))$ be the Kuranishi neighborhoods constructed in the previous section and
$$ V(u_1,u_2)= (\ev_1\times \ev_1)^{-1} (\De),$$
where $\De$ is the diagonal in $L\times L$ and $\ev_{1} \colon V(u_i) \to L$ are the evaluation maps. 
For $V(u_1,u_2)$ to be a manifold, we need  $\ev_1\times \ev_1$ to be a submersion. We can choose $\mc{E}_{u_i}$ big enough so that both evaluation maps $\ev_1$ are submersions. For a fixed $\beta$, there are only finitely many topological types of nodal maps which can appear in the limit, and so by induction we can choose obstruction bundles at each step big enough so that the induced Kuranishi structures on the corners of the moduli space for $\beta$ obey the required conditions. Thus, with a suitable choice of obstruction bundles $\mc{E}_{u_i}$, $V(u_1,u_2)$ is a smooth manifold with projections 
$$ \pi_1,\pi_2 \colon V(u_1,u_2) \to V(u_1), V(u_2).$$
Then $\mc{E}_{f_1,f_2}= \pi_1^{-1}\mc{E}_{f_1}\oplus \pi_2^{-1}\mc{E}_{f_2}$ gives the fiber of the corresponding obstruction bundle over $V(u_1,u_2)$ and the Kuranishi map is as before. Thus, we get a Kuranishi neighborhood $(V(u_1,u_2),E(u_1,u_2))$ of $u$ in the boundary component
$$\mc{M}^{*,\tn{reg},\disc}_{1,0}(X,L,D,\rho_1,\beta_1)\times_{(\ev_1,\ev_1)} \mc{M}^{*,\tn{reg},\disc}_{1,0}(X,L,D,\rho_2,\beta_2).$$
In order to extend this Kuranishi neighborhood to a Kuranishi neighborhood $V(u)= V(u_1,u_2) \times [0,\ep)$ of $u$ in the original moduli space, we glue the domain and deform the nodal maps in $V(u_1,u_2)$ into $J$-holomorphic discs modulo obstruction.

Let $z_1,z_2$ be local coordinates near $1\in D^2$, modeled on the closure of upper half-plane as neighborhoods of  $0\in \H$. For each positive real gluing parameter $\mu$, consider the Riemann surface $\Si_{\mu}\cong D^2$ obtained by gluing $\Si_1$ and $\Si_2$ via $z_1z_2=-\mu$. This gluing respects the orientation of the boundary on each part. Since the divisor $D$ is disjoint from $L$, a straightforward modification of the proof of \cite[Proposition 7.2.12]{FOOO} yields the following.
 
\begin{proposition}\label{gluing}
There is a continuous family of embeddings 
$$\iota_{\mu} \colon V(u_1,u_2) \to W^{1,p}(X,\beta),\quad \mu \in (0,\ep),$$ 
with the following properties: 
\begin{enumerate}
\item $f_{\mu}=\iota_{\mu}(f_1,f_2)$ converges to $(f_1,f_2)$ as $\mu \rightarrow 0$;
\item $\bar{\partial}_{J}f_{\mu} \in \mc{E}_{f_1,f_2}$, where $\mc{E}_{f_1,f_2}$ is a subspace of $W^{l-1,p}(E_{f_{\mu}}^{0,1})$ obtained via parallel translation;
\item every map $f'$ close enough to some $f\in V(u_1,u_2)$ with $\bar{\partial}_{J}f' \in \mc{E}_{f'}$ is in the image of some $\iota_{\mu}$.
\end{enumerate}
\end{proposition} 

\subsection{Non $D$-regular maps}

We now consider the case $[u,\Si]$ is not $D$-regular. This means some component of $u$ is mapped into the divisor $D$. As we will see, the part mapped into $D$ satisfies certain properties, and not every such map can be a limit of a $D$-regular stable relative maps.
By Section~\ref{relative-open-invariants}, every non $D$-regular map 
$$[u] \in \ov{\mc{M}}^{\disc}(X,L,D,\rho,\beta)\setminus \mc{M}^{\tn{reg},\disc}(X,L,D,\rho,\beta)$$ 
can be modeled as a stable map into the singular space $X[n]$ as in Definition \ref{stable-relative-map}. Therefore, we can write $\Si= \coprod \Si_i$ such that $u_0= u\mid_{\Si_0}$ is a $D$-regular map into $(X,D)$, possibly from a disconnected domain, and $u_i= u\mid_{\Si_i}$, $i>0$, is a $D$-regular map into $(Y_D, D_{0}\cup D_{\infty})$. Thus, 
\begin{equation*}
\begin{split}
[u_0] &\in \mc{M}^{\tn{reg},\disc}(X,L,D,\rho_0,\Gamma_0),\\
[u_i] &\in \mc{M}^{\tn{reg}}(Y_D,D_0\cup D_{\infty},\rho_i^0\cup\rho_i^{\infty},\Gamma_i)/\C^*,\quad i>0,
\end{split}
\end{equation*}
where
\begin{itemize}
	\item $\Gamma_i$ describes the topological type of the domain and the homology class of the image;
	\item $\rho_i^{0}$ describes the intersection pattern of the map with $D_{0} \subset Y_D$ and $\rho_{i}^{\infty}$, $i>0$, describes the intersection pattern of $u_i$ with $D_{\infty} \subset  Y_D$, and $\rho_i^{0}=\rho_{i+1}^{\infty}$;
	\item the $\C^*$-action on the space of maps into $Y_D$ comes from the $\C^*$-action on the $\P^1$ fibers.
\end{itemize}
For simplicity, we assume that $\Si=\Si_0 \cup \Si_1$ and $u_0$ and $u_1$ are not nodal. The general case is an easy extension of this case. From the previous sections, we have Kuranishi neighborhoods $(V(u_i), E(u_i))$ and evaluation maps 
$$ \ev_i= \ev_{\vec{\zeta}_i}\colon V(u_i) \to D^{l(\rho_0)}, $$
where $\rho_0=\rho_{1}^{0\infty}$ is the intersection pattern between two components and $\vec{\zeta}_i$ are contact points with $D$. Let
$$V(u_0,u_1)= (\ev_0\times \ev_1)^{-1} (\De)$$
be the inverse image of the diagonal map. For each $(f_0,f_1)\in V(u_0,u_1)$, let $\mc{E}_{f_0,f_1}=\mc{E}_{f_0}\oplus \mc{E}_{f_1}$. If $V(u_i)$ is big enough, $V(u_0,u_1)$ is a manifold. This way we get a Kuranishi neighborhood of $[u]$ in
$$ \mc{M}^{\tn{reg},\disc}(X,L,D,\rho_0,\Gamma_0)\times_{(\ev_0,\ev_1)} \mc{M}^{\tn{reg}}(Y_D,D_0\cup D_{\infty}, \rho_i^0\cup\rho_i^{\infty},\Gamma_i)/\C^*. $$
A gluing theorem similar to Proposition \ref{gluing} is needed to extend this to a Kuranishi neighborhood of $[u]$ on $\ov{M}^{\disc}(X,L,D,\rho,\beta)$. Contrary to the previous case, the gluing is not unique, and we get a gluing map from some covering space of $V(u_0,u_1)$. This is the space $\mc{M}(\mc{X}_0,\rho,\Gamma_-,\Gamma_+)$ which appears in the statement of Proposition \ref{kur-str-main}.

Let $(f_0,f_1) \in V(u_0,u_1)$. Since the obstruction bundle is supported away from the marked and nodal points, $u_i$ is $J$-holomorphic near the intersection points $\vec{\zeta}_i$ of $u_i$ with $D$. Let $\vec{q}=\ev_0(\vec{\zeta}_0)=\ev_1(\vec{\zeta}_1)$ be the set of intersection points with $D$. By the symplectic sum procedure of Section \ref{ch:surgery}, we can construct a family $\mc{X}$ over some small disc $\De$ whose central fiber is $X\cup_{D} Y_D$ and whose other fibers are isotopic to $X$ itself. Moreover, in this case, $\mc{X}$ is obtained by blowing up $X\times \De$ along $D\times \left\{0\right\}$. We denote  by $J$ to be the complex structure on $\mc{X}$.

We choose a set of local coordinate charts on $\mc{X}$ around the points $\vec{q}$ as follows. Fix a $\C$-linear identification of $T_{q_i}D$ with $\C^{n-1}$ and extend it to normal coordinates $(v_i^j)$ around $q_i$ in $D$. Let $L_-$ be the normal bundle of $D$ in $X$ and let $L_+$ be the normal bundle of $D$ in $Y_D$. Identifying $L_-\mid_{q_i}$ with $\C$, taking a direct sum with the dual, and parallel translating along radial lines in the $v$-coordinates, we obtain coordinates $(x,y)\colon (L_{-}\oplus L_{+}) \to \C\oplus \C$. This gives a coordinate chart $(v,x,y)$ near each $q_i$ such that the projection $\mc{X} \ra \De$ is given by 
$$(v,x,y)\ra xy\in L_-\otimes L_+=\C .$$ In these coordinates, the almost complex structure $J$ on $\mc{X}$ agrees with the standard almost complex structure on $\C^{n-1}\oplus \C \oplus \C$ at the origin, and $J$ has the form $J_D\oplus J_{\C} \oplus J_{\C}$ along $D$. By \cite[Lemmas 3.2,3.4]{IP1}, in these coordinates and around each $q_i$, $f\in V(u_0,u_1)$ can be written in the form 
$$ f(z_i,w_i)=f_0(z_i) \# f_1(w_i) = ( h^v(z,w), a_iz_i^{s_i}(1+h^x), b_iw_i^{s_i}(1+h^y)),$$ 
with $h^v(0,0)=h^x(0,0)=h^y(0,0)=0$.
In order to glue the two parts of a map $f$ to get a map $f_{\mu} \in W^{1,p}(X,L)$ with $\bar{\partial}_{J}f_{\mu} \in \mc{E}_{f_{\mu}}$, we first glue the domains. Let $\mu=(\mu_1,\cdots\mu_k)$ be a tuple of (sufficiently small) complex numbers  and define $\Si_{\mu}$ to be the Riemann surface obtained by gluing the domains around the interior marked points $\vec{\zeta}_i$ via the equation $z_jw_j=\mu_j$. So we are replacing the node with a small cylinder described by the gluing parameters $\mu_i$ at each node. If we can glue two parts of $f$ and get a map $f_{\mu}$ as above for small ${\mu}$, then the part of $f_{\mu}$ which is mapped to the neck can be written in the form $f_{\mu}(z_i,w_i)=(v_{\mu_{i}},x_{\mu_{i}},y_{\mu_{i}})$ with $x_{\mu_{i}}y_{\mu_{i}}=\ep$ for some fixed $\ep \in \De$. For small $\ep$, the maps $f_{\mu}$ are closely approximated by $(q_i,a_i z_i^{s_i},b_i w_i^{s_i})$ near the intersection point $q_i$; see \cite[Section 5]{IP2}. So for $f_{\mu}$ to be in $X_{\ep}$, we need 
\begin{equation}\label{gluing-condition}
 a_i b_i \mu_i^{s_i} = \ep .
 \end{equation}
This shows that there are altogether $\left|\rho\right|=\prod s_i$ possibilities for choosing $\mu$ (for a fixed $\ep$), and each choice leads to a different map. The coefficients $a_i$ and $b_i$ are the $s_i$-jets of the components of $f_j$ normal to $D$ at $\zeta_j^{i}$ modulo higher order terms, and so  
$$ a_i \in (T_{\zeta_0^{i}}^*\Si_{0})^{s_i}\otimes L_{-,q_i},\;\;\; b_i \in (T_{\zeta_1^{i}}^*\Si_{1})^{s_i}\otimes L_{+,q_i}. $$
Let $\mathcal{L}_i^{j}$, $i=0,1$, be the relative cotangent bundle to $\Si_{i}$ at $\zeta_i^{j}$. These are complex line bundles over  the Deligne-Mumford moduli space and the leading coefficients are sections 
$$ a_i \in \Gamma((\mathcal{L}_0^{i})^{s_i}\oplus \ev_{0i}^* L_{-})\quad \tn{and}\quad b_i  \in \Gamma((\mathcal{L}_1^{i})^{s_i}\oplus \ev_{1i}^* L_{+}). $$ 
We conclude that $\mu$ is a multisection of the bundle 
$$\bigotimes [(\mathcal{L}_0^{i})^*\otimes (\mathcal{L}_1^{i})^*] \rightarrow V(u_0,u_1).$$
We define $\tilde{V}(u_0,u_1)$ to be the total space of this multisection. This is an \`etale covering of $V(u_0,u_1)$ and we can pull back the obstruction bundle to get a Kuranishi chart $(\tilde{V}(u_0,u_1),E(u_0,u_1))$. 

In order to finish the construction of a Kuranishi structure on $\ov{\mc{M}}^{\disc}(X,L,D,\rho,\beta)$, we need a gluing theorem to extend the Kuranishi neighborhood $(\tilde{V}(u),E(u))$ to a Kuranishi neighborhood of $[u]$ on $\ov{\mc{M}}^{\disc}(X,L,D,J,\beta)$. This gluing theorem is provided by a slight modification (to accommodate obstruction bundles) of the gluing theorem in \cite[Sections 5--8]{IP2}.

\begin{proposition}\label{relative-gluing}
There is a continuous family of orientation-preserving embeddings
$$\iota_{\mu} \colon \tilde{V}(u_0,u_1) \to W^{1,p}(X,\beta)$$
with the following properties: 
\begin{enumerate}
\item $f_{\mu}=\iota_{\mu}(f_0,f_1)$ converges to $(f_1,f_2)$ as $\mu \rightarrow 0$;
\item $\bar{\partial}_{J}f_{\mu}\in \mc{E}_{f_{\mu}}$, where $\mc{E}_{f_{\mu}}$ is the subspace of $E^{0,1}_{f_{\mu}}$ obtained via parallel translation;
\item any map $f'$ close enough to some $f\in \tilde{V}(u)$ with $\bar{\partial}_{J}f'\in\mc{E}_{f'}$ is in the range of some $\iota_{\mu}$.
\end{enumerate}
\end{proposition}


\end{document}